\documentclass[12pt,leqno]{article}
\usepackage{amsfonts}
\usepackage{amsmath, amsthm, amsfonts, amssymb, color,latexsym, enumerate}
\usepackage{mathrsfs}
\usepackage{url}
\usepackage{color}
\usepackage[utf8]{inputenc}
\usepackage[T1]{fontenc}

\usepackage[notcite,notref]{}

\newtheorem{thm}{Theorem}[section]
\newtheorem{cor}[thm]{Corollary}
\newtheorem{lem}[thm]{Lemma}
\newtheorem{prp}[thm]{Proposition}

\newtheorem{rem}[thm]{Remark}
\theoremstyle{definition}

\newcommand{\scr}[1]{\mathscr #1}
\definecolor{wco}{rgb}{0.5,0.2,0.3}

\numberwithin{equation}{section} \theoremstyle{remark}

\newcommand{\ua}{\uparrow}

\newcommand{\mathbbm}[1]{\text{\usefont{U}{bbm}{m}{n}#1}}

\title{{\bf      Wasserstein  Convergence Rates for Empirical Measures of Random Subsequence of $\{n\alpha\}$ }}
\author{
{\bf    Bingyao Wu$^{a)}$ \, \  Jie-Xiang Zhu$^{b)}$  }\\
\footnotesize{ $^{a)}$ School of Mathematics and Statistics, Fujian Normal University, Fuzhou 350007, China}\\
\footnotesize{$^{b)}$ Department of Mathematics, Shanghai Normal University, Shanghai 200234, China }\\
 \footnotesize{  bingyaowu@163.com;  jiexiangzhu7@gmail.com } }

\begin{document}
\allowdisplaybreaks
\def\R{\mathbb R}  \def\ff{\frac} \def\ss{\sqrt}  \def\B{\mathbf B} \def\TO{\mathbb T}
\def\I{\mathbb I_{\pp M}}\def\p<{\preceq}
\def\N{\mathbb N} \def\kk{\kappa} \def\m{{\bf m}}
\def\ee{\varepsilon}\def\ddd{D^*}
\def\dd{\delta} \def\DD{\Delta} \def\vv{\varepsilon} \def\rr{\rho}
\def\<{\langle} \def\>{\rangle} \def\GG{\Gamma} \def\gg{\gamma}
  \def\nn{\nabla} \def\pp{\partial} \def\E{\mathbb E}
\def\d{\text{\rm{d}}} \def\bb{\beta} \def\aa{\alpha} \def\D{\scr D}
  \def\si{\sigma} \def\ess{\text{\rm{ess}}}
\def\beg{\begin} \def\beq{\begin{equation}}  \def\F{\scr F}
\def\Ric{{\rm Ric}} \def\Hess{\text{\rm{Hess}}}
\def\e{\text{\rm{e}}} \def\ua{\underline a} \def\OO{\Omega}  \def\oo{\omega}
 \def\tt{\tilde}
\def\cut{\text{\rm{cut}}} \def\P{\mathbb P} \def\ifn{I_n(f^{\bigotimes n})}
\def\C{\scr C}      \def\aaa{\mathbf{r}}     \def\r{r}
\def\gap{\text{\rm{gap}}} \def\prr{\pi_{{\bf m},\varrho}}  \def\r{\mathbf r}
\def\Z{\mathbb Z} \def\vrr{\varrho} \def\ll{\lambda}
\def\L{\scr L}\def\Tt{\tt} \def\TT{\tt}\def\II{\mathbb I}
\def\i{{\rm in}}\def\Sect{{\rm Sect}}
\def\M{\scr M} \def\Q{\mathbb Q} \def\texto{\text{o}} \def\LL{\Lambda}
\def\Rank{{\rm Rank}} \def\B{\scr B} \def\i{{\rm i}} \def\HR{\hat{\R}^d}
\def\to{\rightarrow}\def\l{\ell}\def\iint{\int}
\def\EE{\scr E}\def\Cut{{\rm Cut}}\def\W{\mathbb W}
\def\A{\scr A} \def\Lip{{\rm Lip}}\def\S{\mathbb S}
\def\BB{\mathbb B}\def\Ent{{\rm Ent}} \def\i{{\rm i}}\def\itparallel{{\it\parallel}}
\def\g{{\mathbf g}}\def\Sect{{\mathcal Sec}}\def\T{\mathcal T}\def\V{{\bf V}}
\def\PP{{\bf P}}\def\HL{{\bf L}}\def\Id{{\rm Id}}\def\f{{\bf f}}\def\cut{{\rm cut}}
\def\Ss{\mathbb S}
\def\BL{\scr A}\def\Pp{\mathbb P}\def\Pp{\mathbb P} \def\Ee{\mathbb E}

\maketitle

\begin{abstract}
Fix an irrational number $\alpha$. Let $X_1,X_2,\cdots$ be independent, identically distributed, integer-valued random variables with characteristic function $\varphi$, and let $S_n=\sum_{i=1}^n X_i$ be the partial sums. Consider the random walk $\{S_n \alpha\}_{n\ge 1}$ on the torus, where $\{\cdot\}$ denotes the fractional part. We study the long time asymptotic behaviour of the empirical measure of this random walk to the uniform distribution under the general $p$-Wasserstein distance. Our results show that the Wasserstein convergence rate depends on the Diophantine properties of $\alpha$ and the H\"older continuity of the characteristic function $\varphi$ at the origin, and there is an interesting critical phenomenon that will occur. The proof is based on the PDE approach developed by L. Ambrosio, F. Stra and D. Trevisan in \cite{AST} and the continued fraction representation  of the irrational number $\alpha$.
 \end{abstract} \noindent
 \vskip 2cm

\section{Introduction and Main Results}

The circle group $\mathbb{T}$ is defined by $\R / \Z$, which is naturally identified with the interval $[0, 1)$.
 Let $\alpha$ be an irrational number. Throughout this paper, for $x \in \R$ we write
\begin{align*}
\lfloor x \rfloor : = \max_{n \in \Z, \, n \leq x} n, \quad \{ x \} : = x - \lfloor x \rfloor, \quad \| x \| : = \min_{n \in \Z} |x - n| = \min \{\{ x \}, 1 - \{ x \}  \}.
\end{align*}
The famous Weyl's criterion \cite{Weyl} implies that the Kronecker sequence $\big\{ \{j \alpha\} ; j \in \N \big\} \subset \mathbb{T}$ is uniformly distributed, namely
\begin{align} \label{uniform distributed}
\lim_{n \to \infty} \frac{1}{n} \sum_{j = 1}^{n} \mathbbm{1}_{[a, b)} (\{ j \alpha \}) = b - a
\end{align}
holds for every $0 \leq a < b \leq 1$, where $\mathbbm{1}_{[a, b)}$ is the indicator function of the interval $[a, b)$. Indeed, \eqref{uniform distributed} is equivalent to the fact that the empirical measures $\frac1n \sum_{j = 1}^{n} \delta_{\{ j \alpha \}}$ weakly converge to the uniform distribution on $\mathbb{T}$, denoted by $\mu$. Let $\mathscr{P}(\mathbb{T})$ be the set of all Borel probability measures on $\mathbb{T}$. It has been a notable question of interest in probability, number theory and dynamical system to explore the rate of this convergence in terms of specific distances defined on $\mathscr{P}(\mathbb{T})$. A natural distance considered is called the discrepancy. The discrepancy between $\nu_0, \nu_1 \in \mathscr{P}(\mathbb{T})$ is defined by
\begin{align*}
D(\nu_0, \nu_1) : = \sup_{0 \leq a < b \leq 1} \big| \nu_0([a, b)) - \nu_1 ([a, b)) \big|;
\end{align*}
and for any sequence $( x_j )_{j \in \N} \subset \mathbb{T}$, we define for $n \in \N$,
\begin{align*}
D_n(x_j) : = D \bigg( \frac1n \sum_{j = 1}^{n} \delta_{x_j}, \mu \bigg).
\end{align*}
Another class of distances is the Wasserstein distances. Note that the circle group $\mathbb{T}$ is equipped with the distance given by $\| x - y \|$ for every $x, y \in [0, 1)$. Then for $p \geq 1$, the $L^p$-Wasserstein distance (cf. \cite{V}) between $\nu_0, \nu_1 \in \mathscr{P}(\mathbb{T})$ is defined by
$$
\W_p(\nu_0, \nu_1) := \bigg( \inf_{\varpi} \int_{\mathbb{T} \times \mathbb{T}} \| x - y \|^p \, \d \varpi (x, y)  \bigg)^{\frac1p},
$$
where the infimum is taken over all couplings $\varpi$ of $\nu_0$ and $\nu_1$. Some connections between the discrepancy and the Wasserstein distance have been discussed in \cite{BB2, Graham}.

For an irrational number $\alpha$, the type $\gamma$ of $\alpha$ (cf. \cite{KN}) is defined by
\begin{align*}
\gamma : = \sup \big\{ \eta \, ; \, \liminf_{q \to \infty} q^\eta \| q \alpha \| = 0   \big\} = \inf \big\{ \eta \, ; \, \inf_{q \in \Z^+} q^\eta \| q \alpha \| > 0  \big\}.
\end{align*}
This quantity reflects the extent to which $\alpha$ can be approximated by rationals, and according to Dirichlet's theorem, $\gamma \geq 1$. If $\gamma$ is finite, then for any $\varepsilon > 0$, there exists a constant $C = C(\alpha, \varepsilon)>0$ such that
$$
\| q \alpha \| \geq C q^{-\gamma - \varepsilon}
$$
holds for any $q \in \N$; and
$$
\| q \alpha \| \leq  q^{-\gamma + \varepsilon}
$$
holds for infinitely many $q \in \N$. As stated in the celebrated Roth's theorem (cf. \cite{S91}), all irrational algebraic numbers have type $1$. The order of decay in $D_n(\{ j \alpha \})$ for $\alpha$ of finite type is presented in \cite[Chapter 2]{KN}.
On the other hand, recent developments on the order of $\W_p(\frac1n \sum_{j = 1}^{n} \delta_{\{ j \alpha \}}, \mu)$ for badly approximated $\alpha$ (to be defined below), as well as generalizations in higher dimensions, can be found in \cite{BS} and the references therein.

In this paper, we investigate the corresponding problem for random subsequences of $\{ j \alpha \}$, which are generated by random walks on $\Z$. Let $X_1, X_2, \ldots$ be independent, identically distributed, integer-valued random variables, and let $S_j = \sum_{i = 1}^{j} X_i$ be the partial sums. Then $\big\{   \{S_j \alpha\} ; j \in \N \big\}$ gives rise to a random walk on $\mathbb{T}$ starting at the origin with a countably-infinite state space. This type of random walk and its variants have been extensively studied in the literature, including \cite{S84, S98, S99, W04, B18, BB, BB2} to cite a few. The nearly optimal order of decay of the discrepancy $D_n(\{ S_j \alpha \})$ has been determined in \cite{BB} under certain assumptions on $X_1$ and $\alpha$. The main purpose of this paper is to establish the upper and lower bounds of $\E \big[ \W_p(\mu_n, \mu) \big]$ within similar contexts, where $\mu_n := \frac1n \sum_{j = 1}^{n} \delta_{\{ S_j \alpha \}}$ is the associated empirical measure. For the upper bound part, we have the following statement.

\begin{thm}\label{TH1}
Let $X_1,X_2,\ldots$ be integer-valued i.i.d. random variables with characteristic function $\varphi$, and let $S_j = \sum_{i = 1}^{j} X_i$. Let $\alpha$ be an irrational number such that
$$\| q \alpha \| \geq C q^{-\gamma}$$
holds for any $q \in \N$ with some constants $\gamma \geq 1$ and $C > 0$. Suppose that
\begin{equation}\label{Con}
|1-\varphi(x)| \ge c|  x |^{\beta}.
\end{equation}
holds in an open neighborhood of $0$ with constants $c > 0$ and $0< \beta \le 2$. Set $\mu_n = \frac{1}{n} \sum_{j = 1}^{n} \delta_{\{ S_j \alpha \}}$, $n \geq 1$. Then for $1 \leq p < \infty$,
\begin{equation*}
\E \big[\W_p(\mu_n,\mu) \big] =
\begin{cases}
O \big( \frac{1}{\sqrt{n}} \big) , & \textrm{if } \beta \gamma < 2;\\
O\big( \frac{(\log n)^{1 - \frac{1}{p \vee 2}}}{\sqrt{n}}  \big), & \textrm{if } \beta \gamma =2;\\
O\big( \frac{1}{n^{1/(\beta \gamma)}} \big),& \textrm{if } \beta \gamma >2,
\end{cases}
\end{equation*}
where the implied constants depend on $p$ and conditions on $\alpha$ and $\varphi$.
\end{thm}

As for the lower bound part, we have the following theorems, which apply to the cases $\beta \gamma \leq 2$ and $\beta \gamma > 2$ respectively.

\begin{thm}\label{TH2}
Let $X_1,X_2,\ldots$ be integer-valued i.i.d. non-constant random variables, and let $S_j = \sum_{i = 1}^{j} X_i$. Let $\aa$ be an irrational number and set $\mu_n = \frac{1}{n} \sum_{j = 1}^{n} \delta_{\{ S_j \alpha \}}$, $n \geq 1$. Then there exists a constant $c > 0$ such that for any $n \geq 1$,
\begin{equation*}
\E\left[\W_1(\mu_n,\mu)\right] \ge \frac{c}{\sqrt{n}}.
\end{equation*}
\end{thm}

\begin{thm} \label{lower final}
Let $X_1, X_2, \ldots$ be integer-valued i.i.d. random variables with characteristic function $\varphi$, and let $S_j = \sum_{i = 1}^{j} X_i$. Let $\alpha$ be an irrational number such that
$$\| q \alpha \| \leq C q^{-\gamma}$$
holds for infinitely many $q \in \N$ with some constants $\gamma \ge 1$ and $C > 0$. Suppose that
$$|1 - \varphi(x)| \leq c |x|^\beta $$
holds for any $x \in \R$ with constants $c > 0$ and $0 < \beta \leq 2$. Set $\mu_n = \frac{1}{n} \sum_{j = 1}^{n} \delta_{\{ S_j \alpha \}}$, $n \geq 1$. Then there exists $C' > 0$ such that
$$
\E \big[ \W_1(\mu_n, \mu) \big] \geq \frac{C'}{n^{1/(\beta \gamma)}}
$$
holds for infinitely many $n \in \N$. In other words,
$$
\limsup_{n \to \infty} n^{1/(\beta \gamma)} \E \big[ \W_1(\mu_n, \mu) \big] > 0.
$$
\end{thm}

According to the upper bounds in Theorem \ref{TH1} and the lower bounds in Theorems \ref{TH2} and \ref{lower final}, we obtain the main theorem as follows.

\begin{thm} \label{two-sided}
Let $X_1,X_2,\ldots$ be integer-valued i.i.d. random variables with characteristic function $\varphi$, and let $S_j = \sum_{i = 1}^{j} X_i$. Let $\alpha$ be an irrational number. Suppose that
\begin{enumerate}[(i)]
\item there exists $0 < \beta \leq 2$ such that in an open neighborhood of $0$,\footnote{Throughout the paper $A \asymp B$ means that there exists an absolute constant $C > 0$ such that $C^{-1} B \leq A \leq C B$.}
$$
|1 - \varphi(x)| \asymp |x|^\beta;
$$

\item there exists $\gamma \geq 1$ such that
\begin{align} \label{wy}
0 < \liminf_{q \to \infty} q^\gamma \| q \alpha \| < \infty.
\end{align}
\end{enumerate}
Set $\mu_n = \frac{1}{n} \sum_{j = 1}^{n} \delta_{\{ S_j \alpha \}}$, $n \geq 1$. Then for any $1 \leq p < \infty$,
\begin{enumerate}[(1)]
\item if $\beta \gamma < 2$, assume furthermore that $X_1$ is non-constant, then

$$\E \big[  \W_p(\mu_n, \mu)  \big] \asymp \frac{1}{\sqrt{n}};$$

\item if $\beta \gamma = 2$, assume furthermore that $X_1$ is non-constant, then there exist $c_1, c_2 > 0$ such that for any $n \in \N$,

$$\frac{c_1}{\sqrt{n}} \leq \E \big[  \W_p(\mu_n, \mu)  \big] \leq \frac{c_2 (\log n)^{1 - \frac{1}{p \vee 2}} }{\sqrt{n}};$$

\item if $\beta \gamma > 2$, then

$$
0 < \limsup_{n \to \infty} n^{1/(\beta \gamma)} \E \big[ \W_p(\mu_n, \mu) \big] < \infty,
$$
\end{enumerate}
where all constants and implied ones depend on $p$ and conditions on $\varphi$ and $\alpha$.
\end{thm}

Let $W(\gamma)$ denote the set of all irrational numbers satisfying \eqref{wy} with $\gamma \geq 1$. The elements of $W(\gamma)$ all have type $\gamma$. According to Jarn\'ik-Besicovitch theorem \cite{B34}, if $\gamma > 1$, $W(\gamma)$ has the Hausdorff dimension not exceeding $2/(\gamma + 1)$, and therefore is a Lebesgue null set. Nevertheless, based on the theory of continued fractions, each $W(\gamma)$ has the cardinality of the continuum (see Remark \ref{wy0} below).

Theorem \ref{two-sided} indicates that the convergence rate changes as $\beta \gamma$ passes the critical value $2$. In \cite{BB}, the same phenomenon has been pointed out in the almost sure asymptotics of $D_n (\{  S_j \alpha \})$. It can be observed that there is a gap between the upper and the lower bounds for the critical case $\beta \gamma = 2$ in Theorem \ref{two-sided}. However, via the so-called PDE approach, some sharper bounds can be achieved for specific $X_1$ and $\alpha$.
We say that an irrational number $\alpha$ is badly approximable if
\begin{align} \label{badly}
\liminf_{q \to \infty} q \| q \alpha \| > 0.
\end{align}
Equivalently, $\alpha$ is badly approximable if and only if its continued fraction (cf. Section \ref{S2.3} below) has bounded elements $(a_k)_{k \in \N}$. All the quadratic irrationals are badly approximable, and the set of all badly approximable numbers is a Lebesgue null set. For badly approximable $\alpha$, our result reads as follows.
\begin{thm} \label{special case}
Let $X_1, X_2, \ldots$ be integer-valued i.i.d. random variables such that $\E X_1 = 0$ and $0 < \E X_1^2 < \infty$. Let $\alpha$ be a badly approximable irrational number and set $\mu_n = \frac{1}{n} \sum_{j = 1}^{n} \delta_{\{ S_j \alpha \}}$, $n \geq 1$. Then for any $1 \leq p \leq 2$,
$$
\E \big[ \W_p (\mu_n, \mu) \big] \asymp \sqrt{ \frac{\log n}{n}}.
$$
\end{thm}

This exactly corresponds to the case $\beta = 2$ and $\gamma = 1$ in Theorem \ref{two-sided}. Moreover, we can also establish a path-wise lower bound for $\W_1(\mu_n, \mu)$ as follows.

\begin{cor} \label{path-wise}
Under the hypotheses of Theorem \ref{special case}, there exists $c > 0$ such that with probability $1$,
$$
\limsup_{n \to \infty} \sqrt{\frac{n}{\log n}}  \W_1(\mu_n, \mu) \geq c.
$$
\end{cor}
Since $D_n( \{S_j \alpha \} ) \geq \W_1(\mu_n, \mu)$, this result slightly improves the lower bound in \cite{BB}. And by the same argument, we have the following path-wise lower bound for the case $\beta \gamma >2$.

\begin{cor} \label{path-wise'}
Under the hypotheses of Theorem \ref{two-sided}, if $\beta \gamma > 2$, then there exists $c > 0$ such that with probability $1$,
$$
\limsup_{n \to \infty} n^{1/(\beta \gamma)} \W_1(\mu_n, \mu) \geq c.
$$
\end{cor}

To cover the case $\beta \gamma \leq 2$, we present the following result, which follows from the law of the iterated logarithm for random exponential sums (see \cite{BB1}) and the Kantorovich duality.

\begin{cor} \label{path-wise''}
Under the hypotheses of Theorem \ref{TH2}, there exists $c > 0$ such that with probability $1$,
$$
\limsup_{n \to \infty} \sqrt{\frac{n}{\log \log n}}  \W_1(\mu_n, \mu) \geq c.
$$
\end{cor}

Our proof of the above results is based on the PDE and mass transportation approach developed by L. Ambrosio, F. Stra and D. Trevisan in \cite{AST} for investigating the optimal matching problem on 2-dimensional compact manifolds. This methodology involves two key ingredients: the comparison between the $L^p$-Wasserstein distance and the $H^{-1, p}$ Sobolev norm of the density function, and the smoothing procedure via the heat semigroup. The first has also been developed by several authors (see e.g. \cite{L17, Peyre, Graham}). For the simple special case $\mathbb{T}$, we adopt the relationship between the $L^p$-Wasserstein distance and the $L^p$ norm of the antiderivative, which is proved in \cite{Graham}. The PDE approach has been a powerful tool in the study of the rate of convergence to equilibrium in the Wasserstein metric $\W_p$, for some recent developments, see e.g. \cite{L17, AGT, WZ, W2, WB, HMT, Bor}.

\textbf{Structure of the paper} In Section 2, we discuss  equivalent characterizations of our assumptions on characteristic functions $\varphi$ and provide several examples. In Section 3, we introduce the related notation and recall some basic facts, including several properties of the heat semigroup on the torus,  useful upper bounds of $L^p$-Wasserstein distance, and the continued fraction representation of $\alpha$. Section 4 is dedicated to proving the upper bounds. We also obtain the renormalized limit of $\E[\W_2^2(\mu_n,\mu)]$ when $\beta\gamma<2$. In Section 5, we focus on the lower bounds. We obtained some interesting results, such as  a more general result about the lower bound of $\W_1(\mu,\nu)$. Additionally, in Section 5.3.1, we present a new method for computing the moment estimation of the $H^{-1,p}$ Sobolev norm of the density function, which differs from the approach in \cite[Proposition 3.6]{HMT}. At the end of this paper, based on the moment estimation of $\W_p(\mu_n,\mu)$ and the Hewitt-Savage zero-one law, we derive some path-wise lower bounds for $\W_1(\mu_n,\mu)$.

\section{More about the assumptions on $\varphi$ and examples}

Let us begin with the assumption on $\varphi$ in Theorem \ref{TH1}:
\begin{equation}\label{condition00}
|1-\varphi(x)| \ge c|  x |^{\beta},
\end{equation}
which holds in an open neighborhood of $0$ with constants $c > 0$ and $0< \beta \le 2$. For integer-valued random variables, \eqref{condition00} has the following equivalent characterizations.

\begin{prp} \label{equivalent}
Let $\varphi$ be a characteristic function of an integer-valued random variable $X$ and let $0< \beta \leq 2$. The following assertions are equivalent.
\begin{enumerate}[(i)]
\item \label{item1}
There exists $c > 0$ such that in an open neighborhood of $0$,
\begin{align*}
|1 - \varphi(x)| \geq c |x|^\beta.
\end{align*}

\item \label{item2}
There exist $d \in \Z^+$ and $c' > 0$ such that for any $x \in \R$,
\begin{align*}
|1 - \varphi(2 \pi x)| \geq c' \| d x \|^\beta.
\end{align*}

\item \label{item3}
There exist $d \in \Z^+$ and $c' > 0$ such that for any $x \in \R$,
\begin{align*}
|1 - \varphi(2 \pi x)| \geq c' |\varphi(2\pi x)| \| d x \|^\beta.
\end{align*}

\end{enumerate}
\end{prp}

\begin{proof}
Since $|\varphi| \leq 1$, the implication from \eqref{item2} to \eqref{item3} is immediate. Combining the facts that $\| d x \| = d |x|$ for $|x| \leq 1/(2d)$ and $\lim_{x \to 0} \varphi(2\pi x) = 1$, we settle the implication from \eqref{item3} to \eqref{item1}. Hence it suffices to address the implication from \eqref{item1} to \eqref{item2}.

For the integer-valued random variable $X$, we define the support of $X$ as follows
$$
\text{supp} (X) : = \big\{ m \in \Z \, ; \, \mathbb{P}( X = m) > 0   \big\}.
$$
Let $d_X$ be the greatest common divisor of $\text{supp}(X)$. Then its characteristic function reads
$$
\varphi(2\pi x) = \sum_{m \in \text{supp} (X)} \mathbb{P}(X = m) \, \e^{2\pi \i m x }.
$$
The condition \eqref{item1} excludes the possibility $\varphi \equiv 1$, therefore $d_X \in \Z^+$ is well-defined. And it is easily checked that $1/d_X$ is a period of $\varphi(2\pi x)$.

Let $Z_\varphi$ be the set of all the solutions of $\varphi(2\pi x) = 1$. Since $|\mathrm{Re}(\e^{2\pi \i m x})| = |\cos (2\pi m x)| \leq 1$ and $\sum_{m \in \Z} \mathbb{P}(X = m) = 1$,  by considering separately the real and imaginary parts, it is known that
\begin{align}
Z_\varphi = \big\{  x \in \R \, ; \, \e^{2\pi \i m x} = 1 \text{ for any } m \in \text{supp}(X) \big\} = \big\{ k / d_X ; k \in \Z  \big\}.
\end{align}

To establish \eqref{item2} with $d = d_X$, since $\varphi$ is periodic, we only need to show that for every $x_0 \in Z_\varphi = \{ k/d_X ; k \in \Z  \}$, there exists $c = c(x_0) > 0$ such that in an open neighborhood of $x_0$,
\begin{align} \label{every point}
|1 - \varphi(2 \pi x)| \geq c |x - x_0|^\beta.
\end{align}
In view of \eqref{item1} and the periodicity, \eqref{every point} immediately follows, which completes the proof.
\end{proof}

In analogy with Proposition \ref{equivalent} \eqref{item2}, the authors of \cite{BB} explored upper bounds for the discrepancy $D_n$ under the following more stringent requirements
\begin{equation} \label{strong1}
1 - |\varphi(2\pi x)| \geq c \| d x \|^\beta
\end{equation}
and
\begin{equation} \label{strong2}
| \varphi(2\pi x) - \varphi(2\pi y)| \geq c \| d (x - y) \|.
\end{equation}
In \cite[Section 3.2]{BB}, they present some concrete examples satisfying these conditions, for instance, specific random variables with heavy tails. It is evident that the framework of Theorem \ref{TH1} covers these examples.

For integrable $X$, its characteristic function $\varphi$ belongs to $C^1(\R)$ and satisfies $\varphi'(0) = \i \E X$. Then it is easy to see that $\varphi$ satisfies \eqref{condition00} with $\beta = 1$ if and only if $\E X \ne 0$. However, it is noteworthy that the integrability of $X$ is not required for the validity of \eqref{condition00} with $\beta = 1$. For instance, consider $X$ with $\mathbb{P}(X = 0) = 0$ and $\mathbb{P}(X = m) = 3/(\pi^2 m^2)$ for any $m \in \Z \backslash \{ 0 \}$. The characteristic function $\varphi$ is explicitly given by
\begin{align*}
\varphi(x) = 1 - \frac{3}{\pi} x + \frac{3}{2\pi^2} x^2, \quad x \in [0, 2\pi].
\end{align*}
Despite $\E |X| = \infty$, it is clear that $\varphi$ satisfies \eqref{condition00} with $\beta = 1$.

Now we discuss the assumption on $\varphi$ in Theorem \ref{lower final}:
\begin{equation}\label{condition000}
|1-\varphi(x)| \leq c|  x |^{\beta}.
\end{equation}
which holds for any $x \in \R$ with constants $c > 0$ and $0< \beta \le 2$. For non-integer values of $\beta$, \eqref{condition000} has the following characterization. For its proof, see \cite[Theorem 11.3.3]{K72}.

\begin{prp}
Let $\varphi$ be a characteristic function of a random variable $X$ and let $0< \beta < 2$ with $\beta \ne 1$. The following assertions are equivalent.
\begin{enumerate}[(i)]
\item There exists $c > 0$ such that for any $x \in \R$,
$$
|1 - \varphi(x)| \leq c |x|^\beta.
$$

\item
$$
\limsup_{x \to \infty} x^\beta \mathbb{P}(|X| \geq x) < \infty.
$$
\end{enumerate}
\end{prp}
As for the particular case $\beta = 2$, \cite[Theorem 11.2.1]{K72} implies that \eqref{condition000} holds if and only if $\E X = 0$ and $\E X^2 < \infty$.

\section{Preliminaries}

\subsection{Fourier analysis and heat semigroup on $\mathbb{T}$}

The circle group $\mathbb{T}$ is a compact abelian Lie group, whose Pontryagin dual is $\Z$. The Haar measure on $\mathbb{T}$ is the restriction of Lebesgue measure, denoted by $\mu$. In the sequel, we write $L^p(\mathbb{T}) = L^p(\mathbb{T}, \mu)$ with associated norm $\| \cdot \|_p$, $1 \le p \le \infty$. Since $\mathbb{T}$ is identical to the unit circle, each function on $\mathbb{T}$ can be viewed as a $1$-periodic function $f$ on $\R$. Furthermore, if $f$ is sufficiently smooth, then it can be expanded in Fourier series as follows:
\begin{align} \label{fourier expansion}
f(x) = \sum_{m \in \Z} \widehat{f}(m) \e^{2 \pi \i m x}, \quad x \in \mathbb{T},
\end{align}
where
$$
\widehat{f}(m) = \int_{\mathbb{T}} f(y) \e^{-2 \pi \i m y} \,  \d \mu(y)
$$
is the $m$-th Fourier coefficient of $f$. We denote by $\ell^p$ the space $L^p(\Z)$ endowed with counting measure. It is well-known that the decay of Fourier coefficients reflects the regularity of the function. For instance, since $( \e^{2 \pi \i m  x} )_{m \in \Z}$ forms an orthonormal basis of $L^2(\mathbb{T})$, then by Parseval's relation, for any  $( c_m )_{m \in \Z} \in \ell^2$,
\begin{align*}
\bigg\| \sum_{m \in \Z} c_m \e^{2 \pi \i m  x} \bigg\|_2 = \big\|   c_m  \big\|_{\ell^2}.
\end{align*}
On the other hand, it is obvious that for any $( c_m )_{m \in \Z} \in \ell^1$,
\begin{align*}
\bigg\| \sum_{m \in \Z} c_m \e^{2 \pi \i m  x} \bigg\|_\infty \leq \big\|   c_m  \big\|_{\ell^1}.
\end{align*}
By interpolation, we have the following Hausdorff-Young inequality: for $2 \leq p \leq \infty$ and any $( c_m )_{m \in \Z} \in \ell^{p/(p-1)}$,
\begin{align*}
\bigg\| \sum_{m \in \Z} c_m \e^{2 \pi \i m  x} \bigg\|_p \leq \big\|   c_m   \big\|_{\ell^{p/(p-1)}}.
\end{align*}
This inequality combined with \eqref{fourier expansion} yields that for sufficiently smooth $f$,
\begin{align} \label{H-Y}
\| f \|_p \leq \big\|  \widehat{f}(m) \big\|_{\ell^{p/(p-1)}}.
\end{align}

The heat semigroup $(P_t)_{t \ge 0}$ on $\mathbb{T}$ is induced by the heat semigroup on $\R$ acting on $1$-periodic bounded measurable functions, which is defined through
\begin{align} \label{heat semigroup}
P_t f(x) = \int_{\mathbb{T}} f(y) p_t(x, y) \, \d \mu(y), \quad t > 0,
\, x \in \mathbb{T},
\end{align}
where
\begin{align} \label{heat kernel}
p_t(x, y) = \sum_{m \in \Z} \frac{1}{\sqrt{4\pi t}} \e^{-|x - y - m|^2/4t}, \quad x, y \in [0, 1)
\end{align}
is called the heat kernel of $\mathbb{T}$. It is known that for this semigroup, the Lebesgue measure $\mu$ is its invariant measure, and its infinitesimal generator is the Euclidean Laplacian $\Delta$ with the periodic boundary condition. The (complex-valued) eigenfunctions of $\Delta$ are $( \e^{2 \pi \i m  x} )_{m \in \Z}$, i.e.,
$$
\Delta ( \e^{2 \pi \i m  x} ) = - 4 \pi^2 m^2 \e^{2 \pi \i m  x}.
$$
By spectral decomposition, $P_t$ is a Fourier multiplier on $\mathbb{T}$. More precisely, for $f \in L^1(\mathbb{T})$,
\begin{align} \label{multiplier}
\widehat{P_t f}(m) = \e^{-4 \pi^2 m^2 t} \widehat{f}(m), \quad t \geq 0, \, m \in \Z,
\end{align}
which implies that $P_t f$ is smooth for $t > 0$.

The Fourier coefficients of $\nu \in \mathscr{P}(\mathbb{T})$ are defined by
$$
\widehat{\nu}(m) = \int_{\mathbb{T}} \e^{-2 \pi \i m  x} \d \nu(x), \quad m \in \Z,
$$
which is exactly the characteristic function of $\nu$ restricted to the set $2 \pi \Z$. By duality, the heat semigroup $(P_t)_{t \ge 0}$ induces a dual semigroup, which is still denoted by $(P_t)_{t \ge 0}$ for simplicity, on $\mathscr{P}(\mathbb{T})$ via the relation
$$
\int_{\mathbb{T}} f \, \d (P_t \nu) = \int_{\mathbb{T}} P_t f \, \d \nu, \quad f \in C(\mathbb{T}).
$$
Combined with \eqref{heat semigroup}, it follows that for $t>0$, $P_t \nu$ has a smooth density $\int_{\mathbb{T}} p_t(x, \cdot) \, \d \nu(x)$ with respect to $\mu$, and by \eqref{multiplier} its Fourier coefficients satisfy
\begin{align} \label{multiplier'}
\widehat{P_t \nu}(m) = \e^{-4 \pi^2 m^2 t} \widehat{\nu}(m), \quad m \in \Z.
\end{align}
Using Fubini's Theorem, it is easily verified that $\widehat{P_t \nu}(m)$ is also the $m$-th Fourier coefficient of $\int_{\mathbb{T}} p_t(x, \cdot) \, \d \nu(x)$. Consequently, with a slight abuse of notation, we then use the same notation $P_t \nu$ to denote its density.

Another useful concept we will use is the so-called antiderivative. For any $\nu \in \mathscr{P}(\mathbb T)$, we define a left-continuous function $F_\nu$ on $\mathbb T = [0, 1)$ as follows:
\begin{align} \label{def of anti}
F_\nu (x) = \nu\big(  [0, x) \big) - x.
\end{align}
Then $F_\nu$ is an antiderivative of $\nu - \mu$ in the distributional sense. Indeed, for any $g \in C^1(\mathbb T)$, by Fubini’s Theorem,
\begin{align} \label{anti}
\int_{0}^{1} F_\nu(x) g'(x) \, \d \mu(x)  = - \int_{0}^{1} g(x) \,  \d (\nu - \mu)(x),
\end{align}
where we used the fact that $(\nu - \mu)(\mathbb{T}) = 0$. As an application of
\eqref{anti}, the Fourier coefficients of $F_\nu$ are given by
\begin{align} \label{anti'}
\widehat{F_\nu}(m) = \begin{cases}
\frac{\widehat{\nu}(m)}{2\pi \i m} , & \textrm{if } m \ne 0;\\
\int_{0}^{1} x \d \nu(x) - \frac12 ,& \textrm{if } m = 0,
\end{cases}
\end{align}
where ${\widehat{\nu}(m)}$ denotes the $m$-th Fourier coefficient of $\nu$.
\subsection{Smoothing via the heat semigroup}

The following lemma is taken from \cite{Graham}, which provides an explicit representation of $\W_p(\nu, \mu)$ (see \cite{L17} for related bounds in a more general context).

\begin{lem} \label{Lemma2.0}

For any $1 \leq p \leq \infty$ and $\nu \in \mathscr{P} (\mathbb{T})$,
\begin{align*}
\W_p(\nu, \mu) = \inf_{y \in \R} \| F_\nu - y \|_p \leq \| F_\nu - \widehat{F_\nu}(0) \|_p,
\end{align*}
where $F_\nu$ is the antiderivative of $\nu - \mu$ given by \eqref{def of anti}. In the particular case $p = 2$,
\begin{align*}
\W_2^2(\nu, \mu) = \inf_{y \in \R} \| F_\nu - y \|_2^2 = \sum_{m \ne 0} \frac{|\widehat{\nu}(m)|^2}{4 \pi^2 m^2}.
\end{align*}
\end{lem}

The next lemma gives an upper bound for $\W_p(\nu, \mu)$ using Fourier coefficients of $\nu$.

\begin{lem}\label{Lemma2.1}
There exists a constant $C > 0$ such that for any $2 \leq p < \infty$ and $\nu \in \mathscr{P}(\mathbb{T})$,
\begin{align} \label{lem2.1'}
\W_p(\nu, \mu) \leq C \inf_{0 < \varepsilon < 1} \bigg[ \sqrt{p} \, \varepsilon^{\frac12} \, + \,   \bigg( \sum_{m \ne 0} \frac{\e^{-m^2 \varepsilon}}{|m|^{\frac{p}{p-1}}}  |\widehat{\nu}(m)|^{\frac{p}{p-1}}\bigg)^{1 - \frac1p}   \bigg].
\end{align}
\end{lem}

\begin{proof}

For a general $\nu \in \mathscr{P}(\mathbb{T})$, which may be singular with respect to $\mu$, we will use the operator $P_t$ to make it smooth. By the triangle inequality, for any $t > 0$,
\begin{align} \label{triangle}
\W_p(\nu, \mu) \leq  \W_p(P_t \nu, \nu) +  \W_p(P_t \nu, \mu).
\end{align}
For the term $\W_p(P_t \nu, \nu)$, by the joint convexity of $\W_p^p$ (cf. \cite{V}),
\begin{align} \label{convex}
\W_p(P_t \nu, \nu) \leq \bigg(  \int_{\mathbb{T}} \W_p^p(P_t \delta_x, \delta_x)   \, \d \nu(x)  \bigg)^{\frac1p} \leq \sup_{x \in \mathbb{T}} \W_p (P_t \delta_x, \delta_x).
\end{align}
For any $x \in \mathbb{T}$, combined with \eqref{heat kernel}, it follows that
\begin{align*}
& \W_p^p (P_t \delta_x, \delta_x) = \int_{\mathbb{T}} \| x - y \|^p p_t(x, y)  \, \d \mu(y)  = \frac{1}{\sqrt{4\pi t}} \int_{0}^{1} \| x - y \|^p \sum_{m \in \Z}  \e^{-|x - y - m|^2/4t} \, \d \mu(y)\\
& \leq \frac{1}{\sqrt{4\pi t}} \sum_{m \in \Z} \int_{0}^{1} | x - y - m|^p   \e^{-|x - y - m|^2/4t} \, \d \mu(y) = \frac{4^{p/2} \Gamma\big( \frac{p + 1}{2} \big) }{\sqrt{\pi}} t^{p/2}.
\end{align*}
Then as a consequence of Stirling's formula, there exists $c > 0$ such that for any $p \ge 2$ and $t > 0$,
\begin{align} \label{part 1}
\sup_{x \in \mathbb{T}} \W_p (P_t \delta_x, \delta_x) \leq c \sqrt{p} \, t^{\frac12}.
\end{align}

Now we turn to the analysis of $\W_p(P_t \nu, \mu)$. As already explained in the preceding section, $P_t \nu$ is absolutely continuous with respect to $\mu$. Combining \eqref{multiplier'}, \eqref{anti'} and Lemma \ref{Lemma2.0}, we obtain
\begin{align*}
\W_p(P_t \nu, \mu) \leq \bigg\|   \sum_{m \ne 0}   \frac{\e^{-4 \pi^2 m^2 t} \widehat{\nu}(m)}{2\pi \i m}  \e^{2\pi \i m x}      \bigg\|_p.
\end{align*}
Then applying the Hausdorff-Young inequality \eqref{H-Y}, we obtain
\begin{align} \label{part 2}
\W_p(P_t \nu, \mu) \leq \frac{1}{2 \pi} \bigg( \sum_{m \ne 0} \frac{\e^{-\frac{4 \pi^2 p}{p - 1} m^2 t}}{|m|^{\frac{p}{p-1}}}  |\widehat{\nu}(m)|^{\frac{p}{p-1}} \bigg)^{1 - \frac1p}.
\end{align}
Combining \eqref{triangle}-\eqref{part 2} and setting $\varepsilon = \frac{4 \pi^2 p}{p - 1} t$, the proof is complete after taking infimum.

\end{proof}

Alternatively, one can use the following $L^p$-Erd\H{o}s-Tur\'an inequality developed in \cite{Graham} to establish a refinement of Lemma \ref{Lemma2.1}. Indeed, by Propositions 2 and 4 of \cite{Graham}, there exists $C > 0$ such that for any $2 \leq p \leq \infty$ and $\nu \in \mathscr{P}(\mathbb T)$,
\begin{align} \label{another1}
\W_p(\nu, \mu) \leq C \inf_{N \in \N} \bigg[ \frac{1}{N + 1} \, + \,   \bigg( \sum_{0 < |m| \leq N} \frac{|\widehat{\nu}(m)|^{\frac{p}{p-1}}}{|m|^{\frac{p}{p-1}}}   \bigg)^{1 - \frac1p}   \bigg].
\end{align}
Now for any $\varepsilon \in (0, 1)$, we choose $N = N(\varepsilon) \in \N$ such that $(N + 1)^{-1} < \varepsilon^{\frac12} \leq N^{-1}$. Then for any $m$ with $0 < |m| \leq N$, it is clear that $m^2 \varepsilon \leq 1$, which implies
\begin{align} \label{another2}
\frac{1}{N + 1} + \bigg( \sum_{0 < |m| \leq N} \frac{|\widehat{\nu}(m)|^{\frac{p}{p-1}}}{|m|^{\frac{p}{p-1}}}   \bigg)^{1 - \frac1p} \leq \varepsilon^{\frac12} + \e \bigg( \sum_{m \ne 0} \frac{\e^{-m^2 \varepsilon}}{|m|^{\frac{p}{p-1}}}  |\widehat{\nu}(m)|^{\frac{p}{p-1}}\bigg)^{1 - \frac1p}.
\end{align}
Combining \eqref{another1} and \eqref{another2} gives Lemma \ref{Lemma2.1}. Note that this approach allows us to extend Lemma \ref{Lemma2.1} to the limit case $p = \infty$ and remove the factor $\sqrt{p}$ in \eqref{lem2.1'}.

In our proof of Lemma \ref{Lemma2.1}, another way to control $\W_p(P_t \nu, \mu)$ is to use the Sobolev inequality on the torus. By \cite[Proposition 1.1.]{BO13}, for any $2 \leq p < \infty$, there exists $C_p > 0$ such that
\begin{align*}
\bigg\|   \sum_{m \ne 0}   \frac{\e^{-4 \pi^2 m^2 t} \widehat{\nu}(m)}{2\pi \i m}  \e^{2\pi \i m x}      \bigg\|_p \leq C_p \bigg( \sum_{m \ne 0} \frac{\e^{-8 \pi^2 m^2 t}}{|m|^{1 + \frac2p}}  |\widehat{\nu}(m)|^2 \bigg)^{\frac12}.
\end{align*}
Combining this with Lemma \ref{Lemma2.1}, we obtain the following improvement of Lemma \ref{Lemma2.1}, which will not be used in the sequel.
\begin{lem}
For any $2 \leq p < \infty$, there exists a constant $c_p > 0$ such that for any $\nu \in \mathscr{P}(\mathbb{T})$,
\begin{align*}
\W_p(\nu, \mu) \leq c_p \inf_{\substack{0 < \varepsilon < 1 \\ A \subset \Z \backslash \{ 0 \}}} \bigg[  \varepsilon^{\frac12} \, + \, \bigg( \sum_{m \in A} \frac{\e^{-m^2 \varepsilon}}{|m|^{\frac{p}{p-1}}}  |\widehat{\nu}(m)|^{\frac{p}{p-1}}\bigg)^{1 - \frac1p} \, + \,  \bigg( \sum_{ \substack{m \in \Z \backslash A \\ m \ne 0}} \frac{\e^{-\frac{2(p-1)}{p}m^2 \varepsilon}}{|m|^{1 + \frac2p}}  |\widehat{\nu}(m)|^2 \bigg)^{\frac12} \bigg].
\end{align*}
\end{lem}

\subsection{Continued fraction representation of an irrational number} \label{S2.3}
We provide in this part some basic facts on the continued fraction representation, we refer the reader to \cite[Chapter 1]{Cas} for more details.
It is known that every real number has a unique continued fraction representation, and this fraction is infinite for irrational numbers. For any $\alpha \in \R \backslash \Q$, let $\alpha = \big[ a_0; a_1, a_2, \ldots \big]$ be its continued fraction representation, where by definition $a_0 \in \Z$ and $a_k \in \Z^+$ for $k \geq 1$. Based on this representation, the convergents $( p_k/q_k )_{k \in \N}$ of $\alpha$ are defined by the following recurrence formulas
\begin{align} \label{recurrence}
\begin{cases}
p_{k + 1} = a_k p_k + p_{k - 1}\\
q_{k + 1} = a_k q_k + q_{k -1}, \qquad k \in \Z^+,
\end{cases}
\end{align}
with initial data $p_0 = 1, q_0 = 0, p_1 = a_0, q_1 = 1$. Then for $k \in \Z^+$, we have
\begin{align*}
\frac{p_k}{q_k}  = \big[ a_0; a_1, a_2, \ldots, a_{k -1} \big].
\end{align*}
And by induction,
\begin{align} \label{interlace}
p_k q_{k - 1} - q_k p_{k - 1} = (-1)^k, \quad k \in \Z^+,
\end{align}
which implies that $p_k/q_k$ is irreducible. One of the theoretical applications of the convergents is to approximate real numbers by rationals with small denominators. Indeed, the convergents $p_k/q_k$ have the following approximation properties
\begin{align} \label{appro0}
\| q \alpha \| \geq \| q_k \alpha \|, \quad 1 \leq q < q_{k + 1}, \, k \in \Z^+;
\end{align}
and
\begin{align} \label{appro}
\frac{1}{q_k + q_{k + 1}} < |q_k \alpha - p_k| = \| q_k \alpha \| \leq \frac{1}{q_{k + 1}}, \quad k \geq 2.
\end{align}
Another useful fact is that $q_k \alpha - p_k = (-1)^{k + 1}|q_k \alpha - p_k|$.
\begin{rem} \label{wy0}
Given $\gamma > 1$, by the construction in the proof of \cite[Theorem 8E]{S91}, for any $K\geq 2$, there exists an irrational number $\alpha$ whose convergent denominators $(q_k)_{k \in \N}$ satisfy
\begin{align*}
K q_k^\gamma \leq q_{k+1} < K q_k^\gamma + q_k, \quad k \geq 2.
\end{align*}
Combined this with \eqref{appro}, it follows that $\lim_{k \to \infty} q_k^\gamma \| q_k \alpha \| = K^{-1}$. On the other hand, for any $q_k \leq q < q_{k+1}$ with $k \geq 2$, by \eqref{appro0} and \eqref{appro}, we have
\begin{align*}
\| q \alpha \| \geq \| q_k \alpha \| > \frac{1}{K q_k^\gamma + 2 q_k } \geq \frac{1}{K q^\gamma + 2q},
\end{align*}
which implies that $\liminf_{q \to \infty} q^\gamma \| q \alpha \| \geq K^{-1}$. Hence, $\liminf_{q \to \infty} q^\gamma \| q \alpha \| = K^{-1}$. Since $K$ can take any value of $[2, \infty)$, it follows that the set $W(\gamma)$ shares the cardinality of the continuum.
\end{rem}

For the further developments, we need the following lemma, whose proof can actually be directly adopted from the argument of Proposition 4.1 in \cite{BB}. However, for the convenience of the reader, we present its proof here.

\begin{lem}\label{Lem2.2}
Given $\theta > 1$ and $\tau > 0$, then for any irrational number $\alpha$ associated with convergents $( p_k/q_k )_{k \in \N}$ and every $k \geq 3$,
\begin{align*}
\sum_{q_k \leq m < q_{k + 1}} \frac{1}{m^{\theta} \| m \alpha \|^\tau} \, = \,  O \bigg( \frac{1}{q_k^\theta \| q_k \alpha \|^\tau} \bigg) \, + \,
\begin{cases}
O\big( q_k^{-(\theta - 1)} \big), & \textrm{if }0 < \tau < 1;\\
O\big( q_k^{-(\theta - 1)} \log q_k \big), & \textrm{if } \tau = 1;\\
O\big( q_k^{-\theta + \tau} \big), & \textrm{if }\tau > 1,
\end{cases}
\end{align*}
where the implied constants only depend on $\theta$ and $\tau$.
\end{lem}

\begin{proof}
We set $\varepsilon_k  = q_k \alpha - p_k$, then as already explained, we have $\varepsilon_k = (-1)^{k + 1} \| q_k \alpha \|$ and $|\varepsilon_k| \leq q_{k + 1}^{-1}$. Notice that
$$
\| m \alpha \| = \left\|   \frac{m p_k}{q_k}  + \frac{m \varepsilon_k}{q_k}              \right\|
$$
and $\| m \varepsilon_k/q_k \| \leq q_k^{-1}$ for $q_k \leq m < q_{k + 1}$. Then it is natural to split the index set $\big\{  m \, ; \, q_k \le m < q_{k + 1} \big\}$ according to the congruence class of $m p_k$ modulo $q_k$. We thus define
\begin{equation*}
\begin{split}
& A = \big\{ q_k \le m < q_{k+1} \, ; \, m p_k \equiv 0
 \pmod{q_k} \big\},\\
& B = \big\{ q_k \le m < q_{k+1} \, ; \, m p_k \equiv (-1)^k \pmod{q_k} \big\},\\
& C =\big\{q_k \le m < q_{k+1} \, ; \, m p_k \not\equiv 0, (-1)^k \pmod{q_k} \big\}.
\end{split}
\end{equation*}

Since $p_k$ and $q_k$ are relatively prime, for every $m \in A$, we have $m = a q_k$ with some $a \in \Z^+$ and $a \leq q_{k + 1}/q_k$. Then
$$
\| m \alpha \| = a \| q_k \alpha \|,
$$
where we used that $q_k \geq 2$ for $k \ge 3$. Consequently,
\begin{align} \label{part1}
\sum_{m \in A} \frac{1}{m^{\theta} \| m \alpha \|^\tau} \leq \sum_{a = 1}^{\infty} \frac{1}{a^{\theta + \tau}} \cdot \frac{1}{q_k^\theta \| q_k \alpha \|^\tau} = O \bigg(  \frac{1}{q_k^\theta \| q_k \alpha \|^\tau} \bigg).
\end{align}

We now turn to the sum over $B$. It follows from \eqref{interlace} that for every $m \in B$, there exists $a \in \Z^+$ such that $m  = a q_k + q_{k - 1}$. Notice that $q_{k + 1} = a_k q_k + q_{k - 1}$, thus
$$
B = \big\{ a q_k + q_{k - 1} \,  ; \, 1 \leq a \leq a_k - 1 \big\}.
$$
Therefore for $m = a q_k + q_{k - 1} \in B$,
\begin{align*}
\| m \alpha \| & = \left\| \frac{(-1)^k}{q_k} + \frac{(-1)^{k + 1} m |\varepsilon_k|}{q_k}  \right\| = \frac{1 - (a q_k + q_{k - 1})|\varepsilon_k|}{q_k}\\
& \geq \frac{(q_{k + 1} - a q_k - q_{k - 1})|\varepsilon_k|}{q_k} = (a_k - a) |\varepsilon_k|,
\end{align*}
where we used $|\varepsilon_k| \leq q_{k + 1}^{-1}$ in the last line. We then have
\begin{align} \label{part2}
\sum_{m \in B} \frac{1}{m^{\theta} \| m \alpha \|^\tau} &\leq \sum_{a = 1}^{a_k - 1} \frac{1}{(a q_k)^\theta \big( (a_k - a)^\tau \| q_k \alpha \|^\tau \big)} \leq \sum_{a = 1}^{\infty} \frac{1}{a^\theta} \cdot \frac{1}{q_k^\theta \| q_k \alpha \|^\tau} \nonumber\\
& = O \bigg(  \frac{1}{q_k^\theta \| q_k \alpha \|^\tau} \bigg).
\end{align}

Finally, we evaluate the sum over $C$. For every $m \in C$, it is easily checked that
$$
\| m \alpha \| = \left\|   \frac{m p_k}{q_k}  + \frac{m \varepsilon_k}{q_k}    \right\| \geq \frac12 \left\| \frac{m p_k}{q_k} \right\|.
$$
Therefore for each $a \in \Z^+$,
\begin{align} \label{sum1}
\sum_{\substack{a q_k \leq m < (a + 1) q_k \\ m \in C}} \frac{1}{m^{\theta} \| m \alpha \|^\tau} \leq \frac{2^\tau}{a^\theta q_k^\theta} \sum_{a q_k < m < (a + 1) q_k} \frac{1}{ \| m p_k/ q_k  \|^\tau}.
\end{align}
Since $p_k$ and $q_k$ are relatively prime, we have
\begin{align} \label{sum2}
\sum_{a q_k < m < (a + 1) q_k} \frac{1}{ \| m p_k/ q_k  \|^\tau} \leq 2 q_k^{\tau} \sum_{j = 1}^{\lfloor q_k/2 \rfloor} \frac{1}{j^\tau} = \begin{cases}
O ( q_k ), & \textrm{if }0 < \tau < 1;\\
O ( q_k \log q_k ), & \textrm{if } \tau = 1;\\
O ( q_k^\tau ), & \textrm{if }\tau > 1,
\end{cases}
\end{align}
Note that this bound is independent of $a \in \Z^+$. Combining \eqref{sum1} and \eqref{sum2} yields
\begin{align} \label{part3}
\sum_{m \in C} \frac{1}{m^{\theta} \| m \alpha \|^\tau} &\leq \frac{2^\tau}{q_k^\theta} \sum_{a = 1}^{a_k} \frac{1}{a^\theta} \sum_{a q_k < m < (a + 1) q_k} \frac{1}{ \| m p_k/ q_k  \|^\tau} \nonumber\\
& = \begin{cases}
O\big( q_k^{-(\theta - 1)} \big), & \textrm{if }0 < \tau < 1;\\
O\big( q_k^{-(\theta - 1)} \log q_k \big), & \textrm{if } \tau = 1;\\
O\big( q_k^{-\theta + \tau} \big), & \textrm{if }\tau > 1,
\end{cases}
\end{align}
where we used the condition $\theta > 1$ so that $\sum_{a = 1}^{\infty} a^{-\theta} < \infty$. Adding \eqref{part1}, \eqref{part2} and \eqref{part3} completes the proof.
\end{proof}

\section{Upper bounds}

\subsection{Proof of Theorem \ref{TH1}} \label{Pf of T1}

By the definition of $\W_p$ and H\"older's inequality,
$$1 \leq p \leq q < \infty \Longrightarrow \W_p(\mu_n, \mu) \leq \W_q(\mu_n, \mu).
$$
Therefore it suffices to consider the case $p \geq 2$. According to Lemma \ref{Lemma2.1} and H\"older's inequality, there exists $c_1 > 0$, depending on $p$, such that for any $0 < \varepsilon < 1$,
\begin{equation}
\begin{split} \label{EWP1}
 \E \big[ \W_p(\mu_n,\mu) \big] & \le c_1
  \varepsilon^{\frac12} \, + \,  c_1 \E \left[ \bigg( \sum_{m \ne 0} \frac{\e^{-m^2 \varepsilon}}{|m|^{\frac{p}{p-1}}}  |\widehat{\mu_n}(m)|^{\frac{p}{p-1}}\bigg)^{1 - \frac1p} \right]  \\
 &\le c_1 \varepsilon^{\frac12} \, + \,  c_1  \bigg( \sum_{m \ne 0} \frac{\e^{-m^2 \varepsilon}}{|m|^{\frac{p}{p-1}}}  \E \big[ |\widehat{\mu_n}(m)|^2 \big]^{\frac p {2(p-1)}}\bigg)^{1-\ff 1 p},
\end{split}
\end{equation}
where
$$\widehat{\mu_n} (m) = \ff 1 n\sum_{j = 1}^n \e^{- 2\pi \i m S_j \aa}, \quad m\in\Z.$$

We then define
\begin{equation}\label{I}
I_1  := \sum_{m \ne 0} \frac{\e^{-m^2 \varepsilon}}{|m|^{\frac{p}{p-1}}}  \E \big[ |\widehat{\mu_n}(m)|^2 \big]^{\frac p {2(p-1)}}.
\end{equation}
To estimate this series, we need the following lemma regarding $\E \big[ |\widehat{\mu_n}(m)|^2 \big]$.
\begin{lem} \label{bound of fourier}
Under the hypotheses of Theorem \ref{TH1}, there exist $d \in \Z^+$ and $c_2 > 0$ such that for any $m \in \Z \backslash \{ 0 \}$ and $n \geq 1$,
$$
\E \big[ |\widehat{\mu_n}(m)|^2 \big] \leq \frac{c_2}{n \|  m  d \alpha \|^\beta}.
$$
\end{lem}
\begin{proof}
By the definition of $\mu_n$,
\begin{align}\label{Hatmu}
\E \big[ |\widehat{\mu_n}(m)|^2 \big] = \E \bigg[  \bigg| \ff 1 n\sum_{j = 1}^n \e^{- 2\pi \i m S_j \aa} \bigg|^2  \bigg] = \frac1n + \frac2{n^2} \mathrm{Re}\bigg[ \sum_{1 \le j_1 < j_2 \le n } \E \big[ \e^{2 \pi \i m (S_{j_2} - S_{j_1}) \alpha} \big]  \bigg].
\end{align}
Since $X_i$'s are i.i.d., we have
\begin{align*}
& \sum_{1 \le j_1 < j_2 \le n } \E \big[ \e^{2 \pi \i m (S_{j_2} - S_{j_1}) \alpha} \big] = \sum_{1 \le j_1 < j_2 \le n} \E \bigg[ \prod_{i = j_1+1}^{j_2} \e^{2 \pi \i m X_i \alpha} \bigg] = \sum_{1 \le j_1 < j_2 \le n } \varphi(2 \pi m \alpha)^{j_2 - j_1}\\
& = \sum_{j_1 = 1}^{n - 1} \varphi(2 \pi m \alpha)^{-j_1} \sum_{j_2 = j_1 + 1}^{n}  \varphi(2 \pi m \alpha)^{j_2}
= \frac{\varphi(2 \pi m \alpha)}{1 - \varphi(2\pi m \alpha)} \sum_{j_1 = 1}^{n - 1} \big( 1 - \varphi(2 \pi m \alpha)^{n - j_1} \big).
\end{align*}
Using the condition \eqref{Con} and Proposition \ref{equivalent} leads to
\begin{align} \label{Norm}
\bigg|  \sum_{1 \le j_1 < j_2 \le n } \E \big[ \e^{2 \pi \i m (S_{j_2} - S_{j_1}) \alpha} \big]   \bigg| \leq \frac{2n}{c \|  m d \alpha  \|^\beta},
\end{align}
holds for some $d \in \Z^+$ and $c > 0$, where we used that $|\varphi| \leq 1$. Combining \eqref{Hatmu} and \eqref{Norm} completes the proof of this lemma.
\end{proof}
\begin{rem}
A further computation gives that (see also \cite[Proposition 2.2]{BB1})
\begin{align} \label{character}
\E \big[ |\widehat{\mu_n}(m)|^2 \big] = \frac{1 - |\varphi(2 \pi m \alpha)|^2}{|1 - \varphi(2 \pi m \alpha)|^2} \cdot \frac1n + \mathrm{Re} \bigg[ \frac{\varphi(2\pi m \alpha)^{n + 1} - \varphi(2\pi m \alpha)}{(1 - \varphi(2\pi m \alpha))^2} \bigg] \cdot \frac{2}{n^2},
\end{align}
which implies that for fixed $m \in \Z \backslash \{ 0 \}$,
\begin{align} \label{limit of fourier}
\lim_{n \to \infty} n \E \big[ |\widehat{\mu_n}(m)|^2 \big] = \frac{1 - |\varphi(2 \pi m \alpha)|^2}{|1 - \varphi(2 \pi m \alpha)|^2}.
\end{align}
\end{rem}

Applying this lemma, it follows that there exists $c_3 > 0$ such that
\begin{equation*}
I_1 \le  c_3 n^{-\frac{p}{2(p - 1)}}  \cdot \sum_{m=1}^\infty \frac{\e^{-m^2 \varepsilon}}{m^{\frac{p}{p-1}}  \|m d \alpha\|^{\frac{\beta p}{2(p-1)}}},
\end{equation*}
where we used the fact that $\| \cdot \|$ is symmetric. Denoting by $(p_k/q_k)_{k \in \N}$ the convergents of $d \alpha$, then Lemma \ref{Lem2.2} implies that there exists $c_4 > 0$ such that
\begin{align} \label{I01}
I_1 \leq  c_3 n^{-\frac{p}{2(p - 1)}}  \cdot \sum_{k = 1}^{\infty} \e^{- q_k^2 \varepsilon}  \sum_{q_k \le m < q_{k + 1}} \frac{1}{m^{\frac{p}{p-1}}  \|m d \alpha\|^{\frac{\beta p}{2(p-1)}} } \leq c_4 n^{-\frac{p}{2(p - 1)}} \big( 1 + I_2 + I_3 \big)
\end{align}
with
\begin{align*}
I_2 : = \sum_{k = 3}^{\infty}  \e^{  - q_k^2 \varepsilon } q_k^{\frac{p(\beta \gamma - 2)}{2(p - 1)}}
\end{align*}
and
\begin{align*}
I_3 : =
\begin{cases}
\sum_{k = 3}^{\infty} \e^{ - q_k^2 \varepsilon } q_k^{-\frac{1}{p - 1}}, & \textrm{if } 0 < \beta < \frac{2(p - 1)}{p};\\
\sum_{k = 3}^{\infty} \e^{- q_k^2 \varepsilon } q_k^{-\frac{1}{p - 1}} \log q_k, & \textrm{if } \beta = \frac{2(p - 1)}{p};\\
\sum_{k = 3}^{\infty} \e^{ - q_k^2 \varepsilon } q_k^{- \frac{p(2 - \beta)}{2(p - 1)}},& \textrm{if } \beta > \frac{2(p - 1)}{p},
\end{cases}
\end{align*}
where we also used the condition that $\| q \alpha \| \geq C q^{-\gamma}$ for any $q \in \N$.

To evaluate $I_2$ and $I_3$, we make the following observation.
By the recurrence formula \eqref{recurrence}, it follows that $q_{k + 1} = a_k q_k+ q_{k - 1} \geq 2 q_{k - 1}$. Consequently, for each $l \in \Z^+$, there are at most $2$ $q_k$'s that belong to the interval $[2^l, 2^{l + 1} )$. Therefore there exists $c_5 > 0$ such that
\begin{align} \label{i2}
I_2 \leq \sum_{l = 1}^{\infty} \sum_{2^l \leq q_k < 2^{l + 1}}\e^{- q_k^2 \varepsilon } q_k^{\frac{p(\beta \gamma - 2)}{2(p - 1)}} \leq c_5 \sum_{l = 1}^{\infty} \e^{ - 4^l \varepsilon} 2^{\frac{p(\beta \gamma - 2)}{2(p - 1)} l}.
\end{align}
The similar bound on $I_3$ is achieved in the same way.

Next we need the following simple lemma.
\begin{lem} \label{lem3.3}
Given $\eta \in \R$, define a function
$$
g_\eta(\varepsilon) = \sum_{l = 1}^{\infty} \e^{ - 4^l \varepsilon} 2^{\eta l}, \quad \varepsilon \in (0, 1].
$$
Then
\begin{align*}
g_\eta(\varepsilon) \asymp
\begin{cases}
1, & \textrm{if } \eta < 0;\\
\log(\varepsilon^{-1} + 1), & \textrm{if } \eta = 0;\\
\varepsilon^{-\eta/2}, & \textrm{if } \eta > 0.
\end{cases}
\end{align*}
\end{lem}
\begin{proof}
The case $\eta < 0$ is obvious. It suffices to consider the case $\eta \geq 0$. Notice that for $\varepsilon \in (0, 1]$,
$$
g_\eta(\varepsilon/4) =  2^\eta \e^{-\varepsilon} + 2^\eta g_\eta(\varepsilon),
$$
so that
$$
 \e^{-1} 2^\eta + 2^\eta g_\eta(\varepsilon) \leq g_\eta(\varepsilon/4) \leq 2^\eta + 2^\eta g_\eta(\varepsilon).
$$
By iteration, we have for $k \in \Z^+$,
\begin{equation}
\begin{split}\label{iteration}
g_\eta(4^{-k}) &\leq 2^{\eta}  +  2^\eta g_\eta(4^{-(k - 1)} \big) \leq 2^{\eta} + 2^{2\eta} + 2^{2\eta} g_\eta(4^{-(k - 2)})\\
& \leq \cdots \leq \sum_{i = 1}^{k} 2^{i \eta} + 2^{k\eta} g_\eta(1)
\end{split}
\end{equation}
and
\begin{equation} \label{iteration'}
g_\eta(4^{-k}) \geq \e^{-1} 2^{\eta}  +  2^\eta g_\eta(4^{-(k - 1)} \big) \geq \cdots \geq \e^{-1} \sum_{i = 1}^{k} 2^{i \eta} + 2^{k\eta} g_\eta(1).
\end{equation}
For any $\varepsilon \in (0, 1)$, there exists $k \in \Z^+$ such that $4^{-k} \leq \varepsilon < 4^{-k+1}$, then $g_\eta(4^{-k+1}) < g_\eta(\varepsilon) \leq g_\eta(4^{-k})$. Combining this with \eqref{iteration} and \eqref{iteration'} completes the proof.
\end{proof}
Applying Lemma \ref{lem3.3} to \eqref{i2} yields
\begin{align} \label{I02}
I_2 =
\begin{cases}
O(1), & \textrm{if } \beta \gamma < 2;\\
O(\log(\varepsilon^{-1} + 1)), & \textrm{if } \beta \gamma = 2;\\
O ( \varepsilon^{-p(\beta \gamma - 2)/4(p - 1)} ), & \textrm{if } \beta \gamma > 2.
\end{cases}
\end{align}
Similarly, we have
\begin{align} \label{I03}
I_3 =
\begin{cases}
O(1), & \textrm{if } 0 < \beta < 2;\\
O ( \log(\varepsilon^{-1} + 1) ), & \textrm{if } \beta = 2.
\end{cases}
\end{align}
Recalling that $\gamma \geq 1$ and combining \eqref{EWP1}, \eqref{I}, \eqref{I01}, \eqref{I02} and \eqref{I03}, we deduce that, for any $\varepsilon \in (0, 1/2)$,
\begin{align*}
\E \big[ \W_p(\mu_n,\mu) \big] =
\begin{cases}
O\big( \varepsilon^{\frac12} \, + \, n^{-\frac12} \big), & \textrm{if } \beta \gamma < 2;\\
O\big( \varepsilon^{\frac12} \, + \, n^{-\frac12} \log(\varepsilon^{-1})^{\frac{p - 1}{p}}  \big), & \textrm{if } \beta \gamma = 2;\\
O \big( \varepsilon^{\frac12} \, + \, n^{-\frac12} \varepsilon^{-\frac{\beta \gamma - 2}{4}} \big), & \textrm{if } \beta \gamma > 2,
\end{cases}
\end{align*}
where the implied constants do not depend on $n$ and $\varepsilon$. The proof is complete by choosing
$
\varepsilon = \varepsilon(n) = \kappa n^{-\frac{2}{\beta \gamma \vee 2}},
$
where $\kappa > 0$ is a small constant.

For the particular case $p = 2$, the upper bounds for $\E \big[\W_2^2(\mu_n,\mu) \big]$ are achieved in the following proposition.

\begin{prp} \label{proposition2}
Under the hypotheses of Theorem \ref{TH1},
\begin{equation} \label{W_2^2}
\E \big[\W_2^2(\mu_n,\mu) \big] =
\begin{cases}
O \big( \frac{1}{n} \big) , & \textrm{if } \beta \gamma < 2;\\
O\big( \frac{\log n}{n} \big), & \textrm{if } \beta \gamma =2;\\
O\big( \frac{1}{n^{2/(\beta \gamma)}} \big),& \textrm{if } \beta \gamma >2.
\end{cases}
\end{equation}
Furthermore, if $\beta \gamma < 2$, then
\begin{align} \label{W_2^2'}
\lim_{n \to \infty} n \E \big[\W_2^2(\mu_n,\mu) \big] = \frac{1}{4 \pi^2} \sum_{m \ne 0} \frac{1 - | \varphi(2\pi m \alpha)|^2}{m^2 |1 - \varphi(2\pi m \alpha)|^2}.
\end{align}
\end{prp}

\begin{proof}
Since \eqref{W_2^2} can be obtained by Lemma \ref{Lemma2.1} and the same argument used above, we only prove \eqref{W_2^2'}.
If $\beta \gamma  < 2$, the argument used in the previous proof shows that
\begin{align} \label{W_2^2''}
\sum_{m \ne 0} \frac{1}{m^2 \| m d \alpha \|^\beta} < \infty.
\end{align}
By Lemma \ref{Lemma2.0},
\begin{align}
n \E \big[\W_2^2(\mu_n,\mu) \big] = \frac{1}{4 \pi^2} \sum_{m \ne 0} \frac{n \E \big[ |\widehat{\mu_n}(m)|^2 \big]}{m^2}.
\end{align}
Combined with Lemma \ref{bound of fourier}, \eqref{limit of fourier} and \eqref{W_2^2''}, the dominated convergence theorem implies the required conclusion.
\end{proof}

\section{Lower bounds}

\subsection{Proof of Theorem \ref{TH2}}

First, by the Kantorovich duality (cf. \cite{V}) and the fact that the functions  $\cos(2\pi x)$ and $\sin(2\pi x)$ are both $1$-Lipschitz, we have
\begin{align}\label{W1L}
\W_1(\mu_n,\mu) \ge \ff{1}{4 \pi n} \bigg( \bigg| \sum_{j=1}^n \cos(2\pi S_j \alpha) \bigg| + \bigg| \sum_{j=1}^n \sin(2\pi S_j \alpha) \bigg| \bigg) \geq \frac{1}{4\pi n} \bigg|   \sum_{j=1}^n \e^{2\pi \i S_j \alpha} \bigg|.
\end{align}

As a consequence of H\"older's inequality, it follows that
\begin{equation}\label{Holder}
\E\bigg[\bigg|\sum_{j=1}^n \e^{2\pi \i S_j \aa} \bigg|^2 \bigg] \leq \E \bigg[\bigg| \sum_{j=1}^n \e^{2\pi \i S_j \aa}\bigg|\bigg]^{2/3} \cdot \E \bigg[ \bigg|\sum_{j = 1}^n \e^{2\pi \i S_j \aa}\bigg|^4\bigg]^{1/3}.
\end{equation}
Next, we need to evaluate the second and fourth moments of the random exponential sum $\sum_{j = 1}^{n} \e^{2\pi \i S_j \alpha}$.
The assumptions on $X_i$ and $\alpha$ imply that $|\varphi(2\pi \alpha)| < 1$ and
$$
\mathbb{P}(4\alpha(X_1 - X_2) \in \Z) < 1.
$$
By \cite[Proposition 2.2]{BB1}, for any integers $p, n \geq 1$ we have
\begin{align}\label{exponential sum}
\bigg|  \E\bigg[\bigg|\sum_{j=1}^n \e^{2\pi \i S_j \aa} \bigg|^{2p} \bigg] - \bigg( \frac{1 - |\varphi(2\pi \alpha)|^2}{|1 - \varphi(2\pi \alpha)|^2}  \bigg)^p p!^2 \binom{n}{p}    \bigg| \leq C_p n^{p-1},
\end{align}
where the constant $C_p$ only depend on $p$ and the common distribution of $X_i$. In particular, by applying \eqref{exponential sum} with $p = 1$ and $p = 2$, we deduce that there exist constants $c_1, c_2 > 0$ such that for any $n \geq 1$,
\begin{equation} \label{exp}
\E\bigg[\bigg|\sum_{j=1}^n \e^{2\pi \i S_j \aa} \bigg|^2 \bigg] \geq c_1 n \quad \textrm{and} \quad
\E\bigg[\bigg|\sum_{j=1}^n \e^{2\pi \i S_j \aa} \bigg|^4 \bigg] \leq c_2 n^2.
\end{equation}
Taking expectation in \eqref{W1L} and combining it with \eqref{Holder} and \eqref{exp}, we complete the proof of Theorem \ref{TH2}.

\begin{rem} \label{remark4.1}
The inequality \eqref{W1L}, which relates $\W_1$ to the exponential sums, serves a similar role as the Koksma's inequality for $D_n$ (cf. \cite[Chapter 2, Corollary 5.1]{KN}). Therefore, by the law of the iterated logarithm for random exponential sums (see \cite[Theorem 1.1]{BB1}), under the conditions of Theorem \ref{TH2},
$$0<\limsup_{n\to\infty}\ff{\left|\sum_{j=1}^n \e^{2\pi\i S_j \alpha}\right|}{\sqrt{n \log\log n}}<\infty\quad {\rm a.s.}$$
Combined with \eqref{W1L}, it follows that there exists a constant $c>0$ such that with probability 1,
$$
\limsup_{n \to \infty} \sqrt{\frac{n}{\log \log n} } \W_1(\mu_n, \mu) \geq c.
$$
\end{rem}

\subsection{Proof of Theorem \ref{lower final}}

Our proof relies on a general result on the lower bound for $L^1$-Wasserstein distance as follows, which might be of independent interest.
\begin{lem} \label{W1 lower-bound}
Let $\big\{ I_i \, ; \, i = 1, \ldots, K \big\}$ be a family of disjoint closed intervals (more precisely, closed arcs) contained in $\mathbb{T}$. The length of each $I_i$ is denoted by $L_i$. Then for any $\nu \in \mathscr{P}(\mathbb{T})$, which is supported in the complement of $\cup_{i = 1}^{K} I_i$ (namely, $\nu ( \cup_{i = 1}^{K} I_i  ) = 0$), we have
\begin{align*}
\W_1(\nu, \mu) \geq \frac{\sum_{i = 1}^{K} L_i^2 }{4}.
\end{align*}
\end{lem}

\begin{proof}
%By the Kantorovich duality (cf. \cite{V}),
%\begin{align*}
%\W_1(\nu, \mu) = \sup_{\| f \|_{\text{Lip}} \leq 1} \bigg| \int_{\mathbb{T}} f  \,\d \nu - \int_{\mathbb{T}} f \, \d \mu \bigg|,
%\end{align*}
%where
%\begin{align*}
%\| f \|_{\text{Lip}} : = \sup_{x \ne y} \frac{|f(x) - f(y)|}{\| x - y \|}.
%\end{align*}
%We define the following class of functions
%\begin{align*}
%\mathcal{F} = \big\{  f \in C(\mathbb{T}) \, ; \, \| f \|_{\text{Lip}}  \leq 1, \text{supp} (f) \subset \cup_{i = 1}^{K} I_i         \big\}.
%\end{align*}
%Then for every $f \in \mathcal{F}$, $f$ vanishes on the support of $\nu$, which implies that
%\begin{align} \label{W1 lower}
%\W_1(\nu, \mu) \geq \sup_{f \in \mathcal{F}} \bigg|  \int_{\mathbb{T}} f \, \d \mu \bigg|.
%\end{align}
For each $i = 1, \ldots, K$, we denote the center of $I_i$ by $x_i$, then $I_i = [x_i - L_i/2, x_i + L_i/2]$. We introduce the ``tent-like'' function of the form
\begin{align*}
f_i(x) := \begin{cases}
x - x_i + L_i/2, & \textrm{if } x \in [x_i - L_i/2, x_i) ;\\
-x + x_i + L_i/2, & \textrm{if } x \in [x_i, x_i + L_i/2);\\
0, & \textrm{otherwise }.
\end{cases}
\end{align*}
A simple computation gives
\begin{align} \label{fi}
\int_{\mathbb{T}} f_i \, \d \mu = \frac{L_i^2}{4}.
\end{align}
Since $I_i$ are disjoint, it is easily verified that
\begin{align*}
\sum_{i = 1}^{K} f_i (x) = \inf_{y \notin \cup_{i = 1}^{K} I_i } \| x - y  \|, \quad x \in \mathbb{T}.
\end{align*}
Consequently, the function $\sum_{i = 1}^{K} f_i$ is $1$-Lipschitz. By the definition of $f_i$ and the assumption on $\nu$, $\sum_{i = 1}^{K} f_i$ vanishes on $\text{supp}(\nu)$.

Combining these facts with the Kantorovich duality yields
\begin{align*}
\W_1(\nu, \mu) = \sup_{\| f \|_{\text{Lip}} \leq 1} \bigg| \int_{\mathbb{T}} f  \,\d \nu - \int_{\mathbb{T}} f \, \d \mu \bigg| \geq \sum_{i = 1}^{K} \int_{\mathbb{T}}  f_i \, \d \mu = \frac{\sum_{i = 1}^{K} L_i^2}{4}.
\end{align*}
\end{proof}

\begin{rem}
Lemma \ref{W1 lower-bound} can be generalized to the setting of Polish spaces. Indeed, let $(M, \rho)$ be a Polish space, on which the $L^1$-Wasserstein distance is defined as usual. Let $O$ be an open subset of $M$, then the function
$$
F_O (x) = \inf_{y \in O} \rho(x, y), \quad x \in M
$$
is $1$-Lipschitz and vanishes on $O$. By the Kantorovich duality, for any probability measures $\nu_1$ and $\nu_2$ on $M$, if $\nu_1$ satisfies $\text{supp}(\nu_1) \subset O$, then
$$
\W_1(\nu_1, \nu_2) \geq \int_M F_O \, \d \nu_2.
$$
\end{rem}

The next lemma provides a connection between the growth of the integer sequence with the lower bound for $L^1$-Wasserstein distance.

\begin{lem} \label{W1 lower-principle}
Let $\alpha$ be an irrational number, and let $q \in \N$ such that  $\| q \alpha \| \leq C q^{-\gamma}$ holds for some constants $\gamma \ge 1$ and $C > 0$.
For simplicity, we set $K(q) = \lfloor q^\gamma / (3C) \rfloor$. Then for any $\nu \in \mathscr{P}(\mathbb{T})$ such that $\text{supp} (\nu) \subset \big\{ \{ j \alpha \} \, ; \, |j| \leq K(q) \big\}$, we have
$$
\W_1(\nu, \mu) \geq \frac{1}{36q}.
$$
\end{lem}

\begin{proof}
Let $p \in \Z$ such that $\| q \alpha \|  = |q \alpha - p|$. Then for every $j \in \Z$ with $|j| \leq K(q)$, by the assumption on $q$,
$$
\left| j \alpha -  \frac{j p}{q} \right| = |j| \frac{\| q \alpha \|}{q} < \frac{1}{3q}.
$$
This inequality indicates that all the points $\{ j \alpha \}$ $(|j| \leq K(q))$ can be covered by $\cup_{k = 0}^{q - 1} \big( \frac{k}{q} - \frac{1}{3q}, \frac{k}{q} + \frac{1}{3q} \big)$, whose complement in $\mathbb{T}$ is the union of $q$ closed intervals, each with a length of $1/(3q)$. Then by Lemma \ref{W1 lower-bound}, we have
$$
\W_1(\nu, \mu) \geq \frac{q}{4} \cdot \frac{1}{9q^2} = \frac{1}{36q}.
$$

\end{proof}

We are now ready to prove Theorem \ref{lower final}. Set $M_n = \max_{1 \leq j \leq n} |S_j|$. Then by the argument used in the proof of \cite[Lemma 6.2]{BB}, under the assumption on $\varphi$, there exist $C_1, C_2 > 0$ such that for all $n \in \N$,
$$
\mathbb{P}\big( M_n < C_1 n^{1/\beta} \big) \geq C_2.
$$

By our assumption on $\alpha$, there exists an infinite set $H_\gamma \subset \N$ such that $\| q \alpha \| \leq C q^{-\gamma}$ for $q \in H_\gamma$. For each $q \in H_\gamma$, we set $N(q) = \lfloor q^{\beta \gamma} / (3 C C_1)^\beta\rfloor$ and define the event $A_q := \{ M_{N(q)} < C_1 N(q)^{1/\beta} \}$. On $A_q$, it is known that for $j = 1, \ldots, N(q)$, $|S_j| \leq \lfloor q^\gamma/(3C)\rfloor$. Applying Lemma \ref{W1 lower-principle} and noticing that $\mathbb{P}(A_q) \ge C_2$, we obtain
\begin{align} \label{LOW}
\E \big[ \W_1(\mu_{N(q)}, \mu) \big] \geq \frac{C_2}{36q}.
\end{align}
Combined with the choice of $N(q)$, the proof is complete.

\subsection{Proof of Theorem \ref{special case}}

As before, we denote by $\varphi$ the characteristic function of $X_1$. By the assumptions on $X_1$, we have
\begin{align} \label{taylor}
\varphi(x) = 1 - \frac{\E X_1^2}{2}x^2 + o(x^2), \quad \textrm{as }  x \to 0,
\end{align}
which implies that \eqref{Con} holds for $\beta = 2$ and some $c > 0$ in an open neighborhood of $0$. On the other hand, by the definition of badly approximable irrationals, there exists $C > 0$ such that for any $q \in \N$,
\begin{align*}
\| q \alpha \| \geq C q^{-1}.
\end{align*}
Therefore the upper bound part of Theorem \ref{special case} is a corollary of Theorem \ref{TH1}, so we only need to consider the lower bound part for $p = 1$. Before proceeding, we list some useful properties of $\alpha$. Since the elements $(a_k)_{k \in \N}$ in the continued fraction of $\alpha$ are bounded, by \eqref{recurrence} and \eqref{appro}, its convergent denominators $(q_k)_{k \in \N}$ satisfy
\begin{enumerate}[(i)] \label{pro of qk}
    \item There exists $K > 1$ such that for any $k \in \Z^+$, $q_{k + 1} \leq K q_k$;

    \item $\| q_k \alpha \| \asymp q_k^{-1}$.
\end{enumerate}

For $\varepsilon \in (0, 1/2)$, we set
\begin{align*}
f_{n, \varepsilon}(x) := \big( P_\varepsilon \mu_n \big) (x) = \sum_{m \in \Z} \e^{-4\pi^2 m^2 \varepsilon} \widehat{\mu_n}(m) \e^{2\pi \i m x}, \quad x \in \mathbb{T},
\end{align*}
which is the density function of $P_\varepsilon \mu_n$ with respect to $\mu$. Applying \cite[lemma 2.1]{HMT} yields that there exists $C> 0$ such that for any $\varepsilon \in (0, 1/2)$ and $M > 0$,
\begin{align} \label{LOWER}
\E \big[ \W_1(\mu_n, \mu)  \big] \geq \frac{1}{M} \E \left[ \big\| \nabla (-\Delta)^{-1}(f_{n, \varepsilon} - 1) \big\|_2^2 \right] - \frac{C}{M^3} \E \left[ \big\| \nabla (-\Delta)^{-1}(f_{n, \varepsilon} - 1) \big\|_4^4 \right],
\end{align}
where $\nabla (- \Delta)^{-1} (P_t \nu - 1)$ is a smooth function and has the following Fourier expansion
\begin{align} \label{F expansion}
\nabla (- \Delta)^{-1} (P_t \nu - 1) (x) = \frac{\i}{2 \pi} \sum_{m \ne 0} \frac{\e^{-4 \pi^2 m^2 t} \widehat{\nu}(m)}{m}  \e^{2\pi \i m x}.
\end{align}
In what follows, we abbreviate
\begin{align*}
\text{A}(n, \varepsilon) := \E \left[ \big\| \nabla (-\Delta)^{-1}(f_{n, \varepsilon} - 1) \big\|_2^2 \right], \quad \text{B}(n, \varepsilon) := \E \left[ \big\| \nabla (-\Delta)^{-1}(f_{n, \varepsilon} - 1) \big\|_4^4 \right].
\end{align*}

By Proposition \ref{equivalent} and \eqref{taylor}, there exist $d \in \Z^+$ and $c > 0$ such that
\begin{align} \label{equivalent'}
|1 - \varphi(2\pi x)| \geq c \| d x \|^2, \quad x \in \R.
\end{align}
Notice that
\begin{align*}
\int_{\mathbb{T}} |\nabla (-\Delta)^{-1} (f_{n, \varepsilon} - 1)|^2 \, \d \mu  = \sum_{m = 1}^{\infty} \frac{\e^{-8 \pi^2 m^2 \varepsilon}}{2 \pi^2 m^2}   |\widehat{\mu_n}(m)|^2.
\end{align*}
Using \eqref{character}, \eqref{equivalent'} and $|\varphi| \leq 1$, we can write
\begin{align*}
\text{A}(n, \varepsilon) = J_1 + J_2,
\end{align*}
where
$$
J_1 = \frac{1}{2 \pi^2 n} \cdot  \sum_{m = 1}^{\infty} \frac{\e^{-8 \pi^2 m^2 \varepsilon}}{ m^2}  \frac{1 - |\varphi(2 \pi m \alpha)|^2}{|1 - \varphi(2 \pi m \alpha)|^2}
$$
and
$$
J_2 =  O \bigg( \frac{1}{n^2} \cdot \sum_{m = 1}^{\infty} \frac{\e^{-8 \pi^2 m^2 \varepsilon}}{m^2  \| m d \alpha  \|^4} \bigg).
$$
By the same argument used in Section \ref{Pf of T1}, we obtain
\begin{align} \label{J2}
J_2 = O(n^{-2} \varepsilon^{-1}).
\end{align}
As for $J_1$, let $(q_k)_{k \in \N}$ be the convergent denominators of $\alpha$, then
\begin{align*}
J_1 \geq \frac{1}{2 \pi^2 n} \cdot \sum_{k = 1}^{\infty} \frac{\e^{-8 \pi^2 q_k^2 \varepsilon}}{ q_k^2} \frac{1 - |\varphi(2 \pi q_k \alpha)|^2}{|1 - \varphi(2 \pi q_k \alpha)|^2}.
\end{align*}
By \eqref{taylor}, it is known that
\begin{align*}
\lim_{x \to 0} \frac{(1 - |\varphi(x)|^2) x^2}{|1 - \varphi(x)|^2} = \frac{4}{\E X_1^2} \in (0, \infty).
\end{align*}
Then it follows from the properties of $(q_k)_{k \in \N}$ that there exist $c_1, c_2 > 0$ such that
\begin{align*}
J_1 \geq \frac{c_1}{n} \cdot \sum_{k = 1}^{\infty} \frac{\e^{-8 \pi^2 q_k^2 \varepsilon}}{ q_k^2 \| q_k \alpha \|^2} \geq \frac{c_2}{n} \cdot  \sum_{k = 1}^{\infty} \e^{-8 \pi^2 K^{2k} \varepsilon}.
\end{align*}
Using the same approach as in the proof of Lemma \ref{lem3.3} with $\eta = 0$, we have
\begin{align} \label{J1}
J_1 \geq \frac{c_3 \log(\varepsilon^{-1})}{n}
\end{align}
for some constant $c_3 > 0$.

\subsubsection{A general moment estimate}

In this sub-section, we aim to study the upper bound for $\text{B}(n, \varepsilon)$. Indeed, we shall establish the following moment estimate in a slightly more general context. The argument developed here can provide an alternative proof for the moment bounds presented in \cite{HMT}.

\begin{prp} \label{proposition1}
Let $X_1,X_2,\ldots$ be integer-valued i.i.d. random variables with characteristic function $\varphi$, and let $S_j = \sum_{i = 1}^{j} X_i$. Let $\alpha$ be an irrational number such that
\begin{align} \label{con''}
\| q \alpha \| \geq C q^{-\gamma}
\end{align}
holds for any $q \in \N$ with some constants $\gamma \geq 1$ and $C > 0$.
Suppose that there exist real numbers $0 < \beta \le 2$, $c > 0$ and an integer $D > 0$ such that for any $x \in \R$,
\begin{equation} \label{con'}
1 - |\varphi(2\pi x)| \ge c \| D x \|^{\beta}.
\end{equation}
Set $\mu_n = \frac{1}{n} \sum_{j = 1}^{n} \delta_{\{ S_j \alpha \}}$, $n \geq 1$. If $\beta \gamma = 2$, then for any $p \in \Z^+$ and
$\kappa n^{- \frac12} \leq \varepsilon < 1/2$, where $\kappa > 0$, there exists $C' > 0$ such that for any $n \geq 2p$,
\begin{align*}
\E \left[ \int_{\mathbb{T}} |\nabla (-\Delta)^{-1} (f_{n, \varepsilon} - 1)|^{2p} \, \d \mu   \right] \leq \frac{C' (\log n)^p}{n^p}.
\end{align*}

\end{prp}

Recalling \eqref{F expansion}, we obtain
$$
(-2\pi \i)^{2p} \int_{\mathbb{T}} |\nabla (-\Delta)^{-1} (f_{n, \varepsilon} - 1)|^{2p} \, \d \mu = \sum_{(m_1, \ldots , m_{2p}) \in \mathbb{I}_p }  \prod_{s = 1}^{2p} \frac{\e^{-4\pi^2 m_s^2 \varepsilon} \widehat{\mu_n}(m_s)}{m_s},
$$
where
$$
\mathbb{I}_p : = \bigg\{ (m_1, \ldots, m_{2p}) \in (\Z \backslash \{ 0 \})^{2p} \, ; \, \sum_{s=1}^{2p} m_s = 0  \bigg\}.
$$
Therefore
$$
\E \left[ \int_{\mathbb{T}} |\nabla (-\Delta)^{-1} (f_{n, \varepsilon} - 1)|^{2p} \, \d \mu   \right] \leq \frac{1}{(2\pi)^{2p}} \sum_{(m_1, \ldots , m_{2p}) \in \mathbb{I}_p } \prod_{s = 1}^{2p} \frac{\e^{-4\pi^2 m_s^2 \varepsilon} }{ |m_s|} \cdot \bigg| \E \bigg[  \prod_{s = 1}^{2p} \widehat{\mu_n}(m_s) \bigg] \bigg|.
$$

Now we need the following lemma.
\begin{lem} \label{lem6.1}
For any $(m_1, \ldots, m_{2p}) \in \Z^{2p}$ and $n \geq 2p$,
\begin{align*}
\bigg| \E \bigg[   \prod_{s = 1}^{2p} \widehat{\mu_n}(m_s)   \bigg] \bigg| \leq \frac{1}{n^{2p}} \cdot \sum_{\sigma \in \mathrm{S}_{2p}} \prod_{s = 1}^{2p} \min \bigg\{ \frac{1}{1 - \big| \varphi \big(2 \pi \sum_{t = s}^{2p} m_{\sigma(t)} \alpha \big) \big|}, n \bigg\}.
\end{align*}
where $\mathrm{S}_{2p}$ denotes the set of all permutations of $\{1, \ldots, 2p\}$.
\end{lem}

\begin{proof}
We can directly calculate as follows:
\begin{align*}
\E \bigg[   \prod_{s = 1}^{2p} \widehat{\mu_n}(m_s)   \bigg] = \frac1{n^{2p}} \cdot \sum_{1\le j_1,\ldots,j_{2p} \le n} \E \big[  \e^{-2 \pi \i \sum_{s = 1}^{2p} m_s S_{j_s} \alpha} \big],
\end{align*}
which implies that
\begin{align} \label{decomposition}
\bigg| \E \bigg[   \prod_{s = 1}^{2p} \widehat{\mu_n}(m_s)   \bigg] \bigg| \leq \frac{1}{n^{2p}} \cdot  \sum_{\sigma \in \mathrm{S}_{2p}} \sum_{1\le j_1 \le \cdots \le j_{2p} \le n}   \left| \E \big[  \e^{-2\pi \i \sum_{s = 1}^{2p} m_{\sigma(s)} S_{j_s} \alpha} \big]      \right|.
\end{align}
In what follows, we work with the case where $\sigma$ is the identity permutation, and note that the other cases can be treated similarly.

Since $X_i$'s are i.i.d. and $S_j - S_{j'} = \sum_{i = j'+1}^{j} X_i$ where $j' \le j$, by induction, we have for $1\le j_1 \le \cdots \le j_{2p} \le n$,
\begin{align*}
 \E \big[  \e^{-2\pi \i \sum_{s = 1}^{2p} m_s S_{j_s} \alpha} \big] &= \varphi(-2 \pi m_{2p} \alpha)^{j_{2p} -j_{2p - 1}} \E \big[  \e^{-2\pi \i (\sum_{s = 1}^{2p - 1} m_s S_{j_s} + m_{2p} S_{j_{2p - 1}} ) \alpha}   \big]\\
& = \cdots = \prod_{s = 1}^{2p} \varphi \big( -2\pi \sum_{t = s}^{2p} m_t \alpha \big)^{j_s - j_{s - 1}},
\end{align*}
where by convention we define $j_0 = 0$. Then we obtain
\begin{align*}
&  \sum_{1\le j_1 \le \cdots \le j_{2p} \le n}  \left| \E \big[  \e^{-2\pi \i \sum_{s = 1}^{2p} m_s S_{j_s} \alpha} \big] \right|
\leq \sum_{1\le j_1 \le \cdots \le j_{2p} \le n} \prod_{s = 1}^{2p} \big| \varphi \big( 2\pi \sum_{t = s}^{2p} m_t \alpha \big) \big|^{j_s - j_{s - 1}}\\
& \leq \sum_{1 \le j_1 \le \cdots \le j_{2p - 1} \le n} \prod_{s = 1}^{2p-1} \big| \varphi \big( 2\pi \sum_{t = s}^{2p} m_t \alpha \big) \big|^{j_s - j_{s - 1}} \sum_{j_{2p} = j_{2p - 1}}^{n} \big|\varphi (2 \pi m_{2p} \alpha)\big|^{j_{2p} - j_{2p - 1}}\\
& \leq \min \bigg\{ \frac{1}{1 - \big| \varphi \big(2 \pi m_{2p} \alpha \big) \big|}, n \bigg\} \cdot \sum_{1 \le j_1 \le \cdots \le j_{2p - 1} \le n} \prod_{s = 1}^{2p-1} \big| \varphi \big( 2 \pi \sum_{t = s}^{2p} m_t \alpha \big) \big|^{j_s - j_{s - 1}}\\
& \leq \cdots \leq \prod_{s = 1}^{2p} \min \bigg\{ \frac{1}{1 - \big| \varphi \big(2 \pi \sum_{t = s}^{2p} m_t \alpha \big) \big|}, n \bigg\}.
\end{align*}
Combined with \eqref{decomposition}, the proof is completed by summing over all $\sigma \in \mathrm{S}_{2p}$.

\end{proof}

Notice that the index set $\mathbb{I}_p$ remains invariant under the action of $\mathrm{S}_{2p}$, combining Lemma \ref{lem6.1} and \eqref{con'}, it follows that there exists $c > 0$, which only depends on $p \in \Z^+$ and the constants in \eqref{con'}, such that
\begin{equation} \label{med}
\begin{split}
& \E \left[ \int_{\mathbb{T}} |\nabla (-\Delta)^{-1} (f_{n, \varepsilon} - 1)|^{2p} \, \d \mu   \right] \\
& \leq \frac{c}{n^{2p}} \cdot \sum_{(m_1, \ldots , m_{2p}) \in \mathbb{I}_p } \prod_{s = 1}^{2p} \frac{\e^{-4\pi^2 m_s^2 \varepsilon}}{|m_s|} \min \bigg\{ \frac{1}{\big\|  \sum_{t = s}^{2p} m_t D \alpha \big\|^\beta }, n \bigg\}.
\end{split}
\end{equation}
Since \eqref{con''} still holds with $\alpha$ replaced by $D\alpha$, we assume for simplicity that $D = 1$.

Next, we need to decompose $\mathbb{I}_p$ in a proper manner. For each $A \subseteq \{1, 2, \ldots, 2p \}$, we define a subset of $\mathbb{I}_p$ as follows
\begin{align*}
\mathbb{I}_p (A) : = \bigg\{ (m_1, \ldots, m_{2p}) \in \mathbb{I}_p \, ; \, \sum_{t = s}^{2p} m_t = 0 \text{ if and only if $s \in A$} \bigg\}.
\end{align*}
We are only concerned with those $A$'s that satisfy $\mathbb{I}_p (A) \ne \emptyset$.
We write $A$ as the form
$$A = \{s_1, \ldots, s_{\lambda}   \},$$
where $s_1 < \cdots < s_\lambda$. By definition, $s_1 = 1$. The fact $\mathbb{I}_p (A) \ne \emptyset$ implies that $A$ contains no consecutive integers, otherwise there exists some $s_v$ such that $m_{s_v} = 0$, which contradicts our definition of $\mathbb{I}_p$. Consequently, for any $1 \leq v \leq \lambda$,
$$
s_{v + 1} - s_v \geq 2,
$$
where we define $s_{\lambda + 1} = 2p + 1$ for convenience. To describe the structure of $\mathbb{I}_p(A)$, we define for $2 \leq N \leq 2p$,
$$
\mathbb{B}_N : = \bigg\{ (m_1, \ldots, m_N) \in (\Z \backslash \{ 0 \})^{N} \, ; \, \sum_{t = 1}^{N} m_t = 0 \text{ and }   \sum_{t = s}^{N} m_t \ne 0 \text{ for } s = 2, \ldots, N   \bigg\}.
$$
For each $(m_1, \ldots, m_{2p}) \in \mathbb{I}_p(A)$, by definition, for any $1 \leq v \leq \lambda$, $\sum_{t = s_v}^{2p} m_t = \sum_{t = s_{v + 1}}^{2p} m_t = 0$, which implies that $\sum_{t = s_v}^{s_{v+1} - 1} m_t = 0$. Furthermore, by construction, $(m_{s_v}, \ldots, m_{s_{v+1} - 1}) \in \mathbb{B}_{s_{v+1} - s_v}$. Therefore, $\mathbb{I}_p(A)$ can be expressed as a direct product of finitely many $\mathbb{B}_N$, namely
$$
\mathbb{I}_p (A) = \mathbb{B}_{N_1} \times \cdots \times \mathbb{B}_{N_\lambda},
$$
where for $1 \leq v \leq \lambda$, $N_v = s_{v+1} - s_v$. We denote by $Q_p$ the set of all the sequences of integers $(N_v)_{v = 1}^\lambda$ such that $\sum_{v = 1}^{\lambda} N_v =2p$ and $N_v \geq 2$. Given any $(N_v)_{v = 1}^\lambda \in Q_p$, it is obvious that the direct product of $\mathbb{B}_{N_v}$'s is contained in $\mathbb{I}_p$ and not empty. Finally, we have
\begin{align} \label{structure}
\mathbb{I}_p = \bigcup_{(N_v)_{v = 1}^\lambda \in Q_p} \mathbb{B}_{N_1} \times \cdots \times \mathbb{B}_{N_\lambda}.
\end{align}
For our purposes, we define the following function
$$
F (N, \varepsilon) : = \sum_{(m_1, \ldots, m_N) \in \mathbb{B}_N}  \prod_{s = 1}^{N} \frac{\e^{-4\pi^2 m_s^2 \varepsilon}}{|m_s|} \prod_{s = 2}^{N} \frac{1}{\| \sum_{t = s}^{N} m_t \alpha \|^\beta},
$$
where $2 \leq N \leq 2p$ and $\varepsilon \in (0, 1/2)$. Combining \eqref{med} and \eqref{structure} yields
\begin{align} \label{med1}
\E \left[ \int_{\mathbb{T}} |\nabla (-\Delta)^{-1} (f_{n, \varepsilon} - 1)|^{2p} \, \d \mu   \right] \leq c \sum_{(N_v)_{v = 1}^\lambda \in Q_p} n^{- (2p - \lambda)} \prod_{v = 1}^{\lambda} F(N_v, \varepsilon).
\end{align}

For the further developments, we need the following two lemmas. The proof of the first one simply uses Lemma \ref{Lem2.2} and the argument presented in Section \ref{Pf of T1} so we omit it.

\begin{lem} \label{lem6.2}
Let $\alpha$ be an irrational number such that $\| q \alpha \| \geq C q^{-\gamma}$ for any $q \in \N$ with some constants $\gamma \geq 1$ and $C > 0$. Given any $\theta > 1$ and $\tau > 0$ which satisfy $\theta = \gamma \tau$, then for $\varepsilon \in (0, 1/2)$,
\begin{align*}
\sum_{m \ne 0} \frac{\e^{-m^2 \varepsilon}}{|m|^\theta \| m \alpha \|^\tau} = O \big( \log(\varepsilon^{-1}) \big).
\end{align*}
\end{lem}

\begin{lem} \label{lem6.3}
Let $\alpha$ be the same as that in Proposition \ref{proposition1} and $\beta = 2/\gamma$. Then,  there exists $c > 0$ such that for any $m_1 \in \Z \backslash \{ 0 \}$ and $\varepsilon \in (0, 1/2)$,
$$
\sum_{m \in \Z \backslash \{ 0, \, m_1 \}} \frac{\e^{-4 \pi^2 m^2 \varepsilon} \cdot \e^{-4 \pi^2 (m_1 - m)^2 \varepsilon}}{|m| |m_1 - m| \| (m_1 - m) \alpha  \|^\beta} \leq  c \varepsilon^{-1} \frac{\e^{- \frac12 \pi^2 m_1^2 \varepsilon}}{|m_1|}.
$$

\end{lem}

\begin{proof}
We split the sum considered into the sets $R_1 : = \{ |m| > |m_1|/2 \} \backslash \{ m_1 \}$, $R_2^+ : = \{ 1 \leq m \leq |m_1|/2 \}$ and $R_2^- : = \{ - |m_1|/2 \leq m \leq -1 \}$. By Lemma \ref{lem6.2} and the condition $\beta \gamma = 2$, there exists $c_1 > 0$ such that
\begin{equation*}
\begin{split}
\sum_{m \in R_1} \frac{\e^{- 4 \pi^2 m^2 \varepsilon} \cdot \e^{- 4 \pi^2 (m_1 - m)^2 \varepsilon}}{|m| |m_1 - m| \| (m_1 - m) \alpha  \|^\beta}
& \leq \frac{2 \e^{- \pi^2 m_1^2 \varepsilon} }{|m_1|} \sum_{m \in R_1} \frac{|m_1 - m| \e^{- 2 \pi^2 (m_1 - m)^2 \varepsilon} \cdot \e^{- 2 \pi^2 (m_1 - m)^2 \varepsilon}}{|m_1 - m|^2 \| (m_1 - m) \alpha  \|^\beta}\\
& \leq  \frac{\varepsilon^{-\frac12}\e^{- \pi^2 m_1^2 \varepsilon}}{\sqrt{2} \pi |m_1|} \sum_{m \ne 0} \frac{\e^{- 2 \pi^2 m^2 \varepsilon} }{ |m|^2 \| m \alpha \|^\beta} \leq  c_1 \varepsilon^{-\frac12} \log(\varepsilon^{-1}) \frac{\e^{- \pi^2 m_1^2 \varepsilon}}{|m_1|},
\end{split}
\end{equation*}
where in the last line we used the inequality $x^{\frac12} \leq \e^x/2$ for $x \geq 0$.

Then we turn to the sum over $R_2^+$. Since $|m_1 - m| \geq |m_1| /2$ on $R_2^+$, we obtain
\begin{align} \label{summ2'}
\sum_{m \in R_2^+} \frac{\e^{- 4 \pi^2 m^2 \varepsilon} \cdot \e^{- 4 \pi^2 (m_1 - m)^2 \varepsilon}}{|m| |m_1 - m| \| (m_1 - m) \alpha  \|^\beta} \leq \frac{2 \e^{- \pi^2 m_1^2 \varepsilon}}{|m_1|} \sum_{m = 1}^{\lfloor |m_1|/2 \rfloor} \frac{\e^{- 4 \pi^2 m^2 \varepsilon}}{m \| (m_1 - m) \alpha  \|^\beta}.
\end{align}
Now we define for $m  = 1, \ldots, \lfloor |m_1|/2 \rfloor$,
\begin{align*}
b_m := \frac{\e^{- 4 \pi^2 m^2 \varepsilon}}{m}, \quad s_m : = \sum_{j = 1}^{m} \frac{1}{\| (m_1 - j) \alpha  \|^\beta},
\end{align*}
and $s_0 = 0$. A direct computation shows that
\begin{equation}
\begin{split} \label{estimate for bm}
0 < b_m - b_{m+1} &= \e^{- 4 \pi^2 m^2 \varepsilon}  \bigg( \frac{1}{m(m+1)} + \frac{1 - \e^{- 4 \pi^2 (2m+1)\varepsilon}}{m+1} \bigg)\\
& \leq \frac{1}{m^2} + 8\pi^2 \varepsilon \e^{- 4 \pi^2 m^2 \varepsilon} \leq \frac{3}{m^2}.
\end{split}
\end{equation}
On the other hand, let $(q_k)_{k \in \N}$ be the convergent denominators of $\alpha$. By Lemma \ref{Lem2.2}, it is known that
$$ \sum_{q_k \leq q < q_{k+1}} \frac{1}{q^2 \| q \alpha \|^\beta} = O(1).
$$
And using the fact that $q_{k + 1} \geq 2 q_{k - 1}$, we deduce that there exists a uniform constant $c_2 > 0$ such that for any $m \in \Z^+$ and $m_1 \in \Z$, the interval $[m_1 - m, m_1 - 1]$ can be covered by no more than $c_2 \log( m + 1)$ intervals of the form $[q_k, q_{k+1})$ or $(-q_{k+1}, -q_k]$. Consequently, there exists $c_3 > 0$, which is independent of $m_1$, such that
\begin{align}  \label{estimate for sm}
s_m \leq m_1^2  \sum_{j = 1}^{m} \frac{1}{(m_1 - j)^2 \| (m_1 - j) \alpha  \|^\beta} \leq c_3 m_1^2 \log (m + 1).
\end{align}
Combining \eqref{summ2'}, \eqref{estimate for bm}, \eqref{estimate for sm} and summing by parts yields
\begin{equation*}
\begin{split}
& \sum_{m \in R_2^+} \frac{\e^{- 4 \pi^2 m^2 \varepsilon} \cdot \e^{- 4 \pi^2 (m_1 - m)^2 \varepsilon}}{|m| |m_1 - m| \| (m_1 - m) \alpha  \|^\beta}\\ &\leq \frac{2 \e^{- \pi^2 m_1^2 \varepsilon}}{|m_1|} \sum_{m = 1}^{\lfloor |m_1|/2 \rfloor} b_m (s_m - s_{m-1}) = \frac{2 \e^{- \pi^2 m_1^2 \varepsilon}}{|m_1|} \bigg(  \sum_{m = 1}^{\lfloor |m_1|/2 \rfloor - 1} (b_{m+1} - b_m) s_m + b_{\lfloor m_1/2 \rfloor} s_{\lfloor m_1/2 \rfloor}   \bigg)\\
& \leq \frac{c_4 \e^{- \pi^2 m_1^2 \varepsilon}}{|m_1|}\big( m_1^2 + |m_1| \log(|m_1| + 1) \big) \leq c_5 \varepsilon^{-1} \frac{ \e^{- \frac12 \pi^2 m_1^2 \varepsilon}}{|m_1|}
\end{split}
\end{equation*}
for some constants $c_4, c_5 > 0$. The sum over $R_2^-$ can be evaluated in the same way. The proof is concluded by adding these estimates together.
\end{proof}

Based on Lemmas \ref{lem6.2} and \ref{lem6.3}, the upper bound  on $F(N, \varepsilon)$ is achieved in the following result.

\begin{lem} \label{estimate for FN}
Given $p \in \Z^+$.
%and $\varsigma \in (0, 1)$.
Then for any $2 \leq N \leq 2 p$ and $\varepsilon \in (0, 1/2)$,
$$
F(N, \varepsilon) = O \big( \varepsilon^{-(N-2)} \log(\varepsilon^{-1})  \big),
$$
where the implied constant only depends on $p$ and the Diophantine approximation properties of $\alpha$.
\end{lem}
\begin{proof}
The case $N =2$ can be derived directly from Lemma \ref{lem6.2}. Indeed,
\begin{align*}
F(2, \varepsilon) = \sum_{m \ne 0} \frac{\e^{-8 \pi^2 m^2 \varepsilon}}{|m|^2 \| m \alpha \|^\beta } = O(\log(\varepsilon^{-1})).
\end{align*}

As for the case $N \geq 3$, it is not hard to see that $\mathbb{B}_N$'s satisfy the following recurrence relation
$$
\mathbb{B}_{N + 1} = \bigg\{ (m_1 - m, m, m_2, \ldots, m_N) \, ; \, (m_1, m_2, \ldots, m_N) \in \mathbb{B}_N, \, m \in \Z \backslash \{ 0, \, m_1 \}  \bigg\}.
$$
Therefore,
\begin{align*}
& F(N+1, \varepsilon)\\
& =  \!\sum_{(m_1, m_2, \ldots, m_N) \in \mathbb{B}_N} \bigg( \sum_{m \in \Z \backslash \{ 0, \, m_1 \}} \frac{\e^{-4 \pi^2 m^2 \varepsilon} \cdot \e^{-4 \pi^2 (m_1 - m)^2 \varepsilon}}{|m| |m_1 - m| \| (m_1 - m) \alpha  \|^\beta}  \bigg)  \prod_{s = 2}^{N} \frac{\e^{-4\pi^2 m_s^2 \varepsilon}}{|m_s|} \prod_{s = 2}^{N} \frac{1}{\| \sum_{t = s}^{N} m_t \alpha \|^\beta},
\end{align*}
where we also used the fact that $m + \sum_{t = 2}^{N} m_t = -(m_1 - m)$. Combined with Lemma \ref{lem6.3}, it follows that
$$
F(N + 1, \varepsilon) \leq c \varepsilon^{-1}  F(N, \varepsilon/8),
$$
which concludes the proof by induction over $N$.
\end{proof}

We are now ready to prove Proposition \ref{proposition1}. Combining Lemma \ref{estimate for FN} and \eqref{med1}, we obtain
\begin{align*}
\E \left[ \int_{\mathbb{T}} |\nabla (-\Delta)^{-1} (f_{n, \varepsilon} - 1)|^{2p} \, \d \mu   \right] \leq c n^{-p} \cdot \sum_{(N_v)_{v = 1}^\lambda \in Q_p} ( \varepsilon^2 n )^{- (p - \lambda)} \log(\varepsilon^{-1})^\lambda.
\end{align*}
Then using the choice of $\varepsilon$ and the fact $\lambda \leq p$, we complete the proof of Proposition \ref{proposition1}.

\subsubsection{Finishing the proof}

Since ${\rm Var}(X_1) > 0$, $X_1$ is non-constant. According to \cite[Proposition 3.2]{BB}, there exist $c > 0$ and $D \in \Z^+$ such that
$$ 1- |\varphi(2\pi x)|\ge c \|D x \|^2, \quad  x\in\R.
$$ 
Then applying Proposition \ref{proposition1} with $p = 2$, there exists $C_1 > 0$ such that
\begin{align*}
\text{B}(n, n^{-\frac12}) \leq  \frac{C_1 (\log n)^2}{n^2}.
\end{align*}
On the other hand, combining \eqref{J2} and \eqref{J1} shows that there exists $C_2 > 0$ such that
\begin{align*}
\text{A}(n, n^{-\frac12}) \geq \frac{C_2 \log n}{n}.
\end{align*}
Then by \eqref{LOWER}, if we choose $M = \Theta \sqrt{\frac{\log n}{n}}$, where $\Theta > 0$ is a sufficiently large constant, then there exists $C_3 > 0$ such that
\begin{align*}
\E \big[ \W_1(\mu_n, \mu)  \big] \geq C_3 \sqrt{\frac{\log n}{n}},
\end{align*}
which completes the proof of Theorem \ref{special case}.

\subsection{Proofs of Corollaries \ref{path-wise}, \ref{path-wise'} and \ref{path-wise''}}

We begin with the following lemma.

\begin{lem}
Let $X_1,X_2,\ldots$ be i.i.d. random variables, and let $\alpha \in \R$. \footnote{Note that in this lemma, $X_1$ is not required to be integer-valued, and $\alpha$ is not required to be irrational.} Set $\mu_n = \frac{1}{n} \sum_{j = 1}^{n} \delta_{\{ S_j \alpha \}}$, $n \geq 1$. Then the event
$$
E : = \bigg\{  \limsup_{n \to \infty}   B_n^{-1} \W_1(\mu_n, \mu) \geq 1    \bigg\}
$$
has probability $0$ or $1$, where $(B_n)_{n \in \N}$ is a sequence of non-negative numbers and satisfies $\lim_{n \to \infty} n B_n = \infty$.
\end{lem}

\begin{proof}
Let $\sigma$ be any finite permutation of $\N$. Then there exists $N_\sigma \in \Z_+$ such that $\sigma(i) = i$ for any $i > N_\sigma$. Consider the permutated sequence $X_{\sigma(1)},X_{\sigma(2)},\ldots$, and let $S_j' = \sum_{i = 1}^{j} X_{\sigma(i)}$ be its partial sums. It is clear that $S_j = S_j'$ for $j \geq N_\sigma$. Set $\mu_n' = \frac{1}{n} \sum_{j = 1}^{n} \delta_{\{ S_j' \alpha \}}$, $n \geq 1$. Then by the definition of $\W_1$, for any $n \geq N_\sigma$,
\begin{align*}
\W_1(\mu_n, \mu_n') \leq \frac1n \sum_{j = 1}^{n} \| S_j \alpha - S_j' \alpha \| \leq \frac{N_\sigma}{2n}.
\end{align*}
Using triangle inequality and the assumption on $(B_n)_{n \in \N}$, we obtain
\begin{align*}
\lim_{n \to \infty} \big| B_n^{-1} \W_1(\mu_n, \mu) - B_n^{-1} \W_1(\mu_n', \mu) \big| = 0.
\end{align*}
Consequently,
\begin{align*}
\limsup_{n \to \infty}   B_n^{-1} \W_1(\mu_n, \mu) = \limsup_{n \to \infty}   B_n^{-1} \W_1(\mu_n', \mu).
\end{align*}
Therefore the event $E$ is invariant by finite permutations of the indices of $(X_i)_{i \in \N}$. By the Hewitt-Savage zero-one law \cite[Theorem 2.5.4]{Dur}, we have $\mathbb{P} (E) = 0$ or $1$.
\end{proof}

The conjunction of Proposition \ref{proposition2} and Theorem \ref{special case} implies that there exists $c \in (0, 1)$ such that
\begin{align} \label{expectation}
\E \big[ \W_1(\mu_n, \mu)  \big] \geq c \sqrt{\frac{\log n}{n}},
\end{align}
and
\begin{align} \label{expectation2}
\E \big[ \W_1^2(\mu_n, \mu)  \big] \leq \E \big[ \W_2^2(\mu_n, \mu)  \big] \leq \frac{c^{-1} \log n}{n}.
\end{align}
Combining \eqref{expectation} and \eqref{expectation2} and then applying the Paley-Zygmund inequality yields
\begin{align*}
\mathbb{P} \bigg( \W_1(\mu_n, \mu) \geq  \frac{c}{2} \sqrt{\frac{\log n}{n}}     \bigg) &\geq \mathbb{P} \bigg( \W_1(\mu_n, \mu) \geq \frac12  \E \big[ \W_1(\mu_n, \mu)  \big]    \bigg)\\
& \geq \frac14 \frac{\big( \E \big[ \W_1(\mu_n, \mu)  \big] \big)^2}{\E \big[ \W_1^2(\mu_n, \mu)  \big]} \geq \frac{c^3}{4} > 0.
\end{align*}
Therefore by Fatou's lemma, we obtain
\begin{align} \label{positivity}
\mathbb{P} \bigg( \limsup_{n \to \infty} \sqrt{\frac{n}{\log n}} \W_1(\mu_n, \mu) \geq  \frac{c}{2}     \bigg) > 0.
\end{align}

Combining \eqref{positivity} and the above lemma with $B_n = \frac{c}{2} \sqrt{\frac{\log n}{n}}$, we complete the proof of Corollary \ref{path-wise}. By the same argument and combining Proposition \ref{proposition2} with Theorem \ref{lower final}, Corollary \ref{path-wise'} follows, and Corollary \ref{path-wise''} has been already proved in Remark \ref{remark4.1}.

It should be noted that the argument developed in \cite{BB} together with the $L^p$-Erd\H{o}s-Tur\'an inequality \eqref{another1} yields some path-wise upper bounds for $\W_p(\mu_n, \mu)$ under the assumptions \eqref{strong1} or \eqref{strong2}. For instance, if $\| q \alpha \| \geq C q^{-\gamma}$
holds for any $q \in \N$ with some $\gamma \geq 1$ and $C > 0$, and \eqref{strong1} holds for some $c > 0$ and $0 < \beta \leq 2$,  then for any $2 \leq p < \infty$, with probability $1$,
\begin{align*}
\W_p(\mu_n,\mu)  =
\begin{cases}
O \bigg( \sqrt{\frac{\log \log n}{n}} \bigg) , & \textrm{if } \beta \gamma < 2;\\
O\bigg( \sqrt{\frac{\log \log n}{n}} (\log n)^{1 - \frac1p}  \bigg), & \textrm{if } \beta \gamma =2;\\
O\bigg( \big(\frac{\log \log n}{n}\big)^{1/(\beta \gamma)} \bigg),& \textrm{if } \beta \gamma >2.
\end{cases}
\end{align*}
We do not intend to discuss path-wise bounds in further detail, but only point out that the lower bound in Corollary \ref{path-wise''} for the case $\beta \gamma < 2$ is optimal.

%\paragraph{Data Availability Statements.} Data sharing not applicable to this article as no datasets were generated or analysed during the current study.

\paragraph{Acknowledgement.} B.W. is supported by the National Key R\&D Program of China (Grant Nos. 2023YFA1010400 and 2022YFA1006003), NSFC (Grant No. 12401174), NSF-Fujian (Grant No. 2024J08051), Fujian Alliance of Mathematics (Grant No. 2023SXLMQN02)
and the Education and Scientific Research Project for Young and Middle-aged Teachers in Fujian Province of China (Grant No. JAT231014). J.-X.Z. is supported by the National Key R\&D Program of China
(Grant Nos. 2022YFA1006000 and 2020YFA0712900), NNSFC (11921001) and Qiming Program of Tianjin University (2023XQM-0001). The authors would like to thank the two anonymous referees for their many insightful comments and questions, which in particular helped us to refine Theorem \ref{special case}.


\begin{thebibliography}{999}

\bibitem{AGT}
L. Ambrosio, M. Goldman, D. Trevisan, \emph{On the quadratic random matching problem in two-dimensional domains}, Electron. J. Probab. 27 (2022), Paper No. 54, 35 pp.

\bibitem{AST} L. Ambrosio, F. Stra,    D. Trevisan, \emph{A PDE approach to a 2-dimensional matching problem}, Probab. Theory  Relat. Fields 173 (2019), 433--477.

\bibitem{BO13} \'A. B\'enyi, T. Oh, \emph{The Sobolev inequality on the torus revisited}, Publ. Math. Debrecen 83 (2013), 359--374.

\bibitem{BB} I. Berkes, B. Borda,
\emph{On the discrepancy of random subsequences of $\{n\aa\}$},
Acta Arith. 191 (2019),  383--415.

\bibitem{BB1}I. Berkes, B. Borda,
\emph{On the law of the iterated logarithm for random exponential sums},
Trans. Amer. Math. Soc. 371 (2019),  3259--3280.


\bibitem{BB2} I. Berkes, B. Borda,
\emph{On the discrepancy of random subsequences of $\{n\aa\}$, II},
Acta Arith. 199 (2021), 303--330.

\bibitem{B34} A. S. Besicovitch, \emph{Sets of fractional dimensions (IV)}. J. London Math. Soc. 9
(1934), 126--131.


\bibitem{B18} S. G. Bobkov,
\emph{Central limit theorem and Diophantine approximations},
 J. Theoret. Probab. 31 (2018), 2390--2411.


%\bibitem{BLG14}  E. Boissard, T. Le Gouic, \emph{On the mean speed of convergence of empirical and occupation
%measures in Wasserstein distance,} Ann. Inst. Henri Poincar\'{e} Probab. Stat. 50 (2014), 539--563.

\bibitem{Bor} B. Borda, \emph{Empirical measures and random walks on compact spaces in the quadratic Wasserstein metric}, Ann. Inst. Henri Poincar\'{e} Probab. Stat. 59 (2023),  2017--2035.


\bibitem {BS} L. Brown,  S. Steinerberger,
\emph{On the Wasserstein distance between classical sequences and the Lebesgue measure},
Trans. Amer. Math. Soc. 373 (2020),  8943--8962.

\bibitem{Cas} J.W.S. Cassels, \emph{An Introduction to Diophantine Approximation}, Cambridge University Press, Cambridge, 1957.

\bibitem {Dur} R. Durrett,  \emph{Probability--theory and examples}, 5th edn. Cambridge University Press, Cambridge, 2019.

\bibitem{Graham} C. Graham, \emph{Irregularity of distribution in Wasserstein distance}. J. Fourier Anal. Appl. 26 (2020),   No. 75, 21 pp.


%\bibitem{Gry} A. Grigor'yan, \emph{Heat Kernel and Analysis on Manifolds,}   American Mathematical Society, Providence, RI; International
%Press, Boston, MA, 2009.

\bibitem{HMT}M. Huesmann, F. Mattesini, D. Trevisan,
\emph{Wasserstein asymptotics for the empirical measure of fractional Brownian motion on a flat torus},
Stochastic Process. Appl. 155 (2023), 1--26.


\bibitem{K72} T. Kawata, \emph{Fourier Analysis in Probability Theory}, Academic Press, 1972.

\bibitem{KN}L. Kuipers, H. Niederreiter, \emph{Uniform Distribution of Sequences}, Pure and Applied Mathematics,  Wiley, New York-London-Sydney, 1974.


\bibitem {L17} M. Ledoux, \emph{On optimal matching of Gaussian samples},  Zap. Nauchn. Sem. S.-Peterburg. Otdel. Mat. Inst. Steklov. (POMI) 457, Veroyatnost' i Statistika. 25 (2017),  226--264.
%\bibitem{W10}  	F.-Y. Wang, \emph{Harnack inequalities on manifolds with boundary and applications,} J. Math. Pures Appl. 94(2010), 304--321.




\bibitem{Peyre} R. Peyre, \emph{Comparison between $\W_2$ distance and $\dot{H}^{-1}$ norm, and localization of Wasserstein
distance},  ESAIM Control Optim. Calc. Var. 24 (2018), 1489--1501.


\bibitem{S84} P. Schatte,
\emph{On the asymptotic uniform distribution of sums reduced mod 1},
 Math. Nachr. 115 (1984), 275--281.


\bibitem{S91} W.M. Schmidt,
\emph{Diophantine Approximations and Diophantine Equations}, Lecture Notes in Mathematics 1467. Springer, Berlin, 1991.

\bibitem{S99} Y. G. Sinai,
\emph{Simple random walks on tori},
 J. Statist. Phys. 94 (1999), 695--708.


\bibitem{S98} F. E. Su,
\emph{Convergence of random walks on the circle generated by an irrational rotation},
Trans. Amer. Math. Soc. 350 (1998), 3717--3741.

%\bibitem{Tal} M. Talagrand, \emph{Upper and lower bounds of stochastic processes}, Modern Surveys in Mathematics 60. Springer-Verlag (2014).


\bibitem{V}
C. Villani,  \emph{Topics in Optimal Transportation}, Graduate Studies in Mathematics, vol. 58, American Mathematical Society, Providence, 2003.

%\bibitem{W14} F.-Y. Wang, \emph{Analysis for Diffusion Processes on Riemannian Manifolds}, World Scientific, Singapore, 2014.

%\bibitem{W1}F.-Y. Wang,  \emph{Precise limit in Wasserstein distance for conditional
%empirical measures of Dirichlet diffusion processes}, J. Funct. Anal. 280 (2021), 108998.

\bibitem{W2}F.-Y. Wang, \emph{Convergence in Wasserstein distance for empirical measures of Dirichlet diffusion processes on manifolds}, J. Eur. Math. Soc. 25 (2023), 3695--3725.

%\bibitem{W3}F.-Y. Wang, \emph{Wasserstein convergence rate for empirical measures on noncompact manifolds}, Stoch. Proc. Appl. 144 (2022), 271--287.

%\bibitem{W4}F.-Y. Wang, \emph{Convergence in Wasserstein distance for empirical measures of semilinear SPDEs}, Ann. Appl. Probab. 33 (2023), 70--84.

\bibitem{WB}
F.-Y. Wang, B. Wu, \emph{Wasserstein convergence for empirical measures of subordinated diffusions on Riemannian manifolds}, Potential Anal. 59 (2023), 933--954.

\bibitem{WZ}F.-Y. Wang, J.-X. Zhu, \emph{Limit theorems in Wasserstein distance for empirical measures of diffusion processes on Riemannian manifolds}, Ann. Inst. Henri Poincar\'{e} Probab. Stat. 59 (2023),  437--475.

\bibitem{W04} M. Weber, \emph{Discrepancy of randomly sampled sequences of reals}, Math. Nachr. 271 (2004), 105--110.

\bibitem{Weyl} H. Weyl, \emph{\"{U}ber die Gleichverteilung von Zahlen mod. Eins}, Math. Ann. 77 (1916), 313--352.



%\bibitem{eWL} L. Wu, Moderate deviations of dependent random variables related to CLT, Ann.  Probab. 23 (1995), 420--445.


 \end{thebibliography}
\end{document}